\documentclass[a4paper, 11pt]{article}

\input xy
\xyoption{all}

\usepackage{amsmath,amsthm,amssymb}
\usepackage[page,header]{append ix}
\usepackage{morefloats}
\usepackage{color}
\usepackage{makeidx}
\usepackage{hyperref}	 	
\usepackage{fancybox}
\usepackage{mathabx}
\usepackage{tikz}
\usetikzlibrary{matrix}

\makeindex 

\usepackage{latexsym, amsmath, amssymb, amsfonts, amscd}
\usepackage{amsthm}
\usepackage{t1enc}
\usepackage[mathscr]{eucal}
\usepackage{indentfirst}
\usepackage{graphicx, pb-diagram}
\usepackage{enumerate}
\usepackage[all,poly,web,knot]{xy}

\newcommand{\Title}{Title}

\numberwithin{equation}{section}

{\theoremstyle{definition}\newtheorem{definition}{Definition}[section]

\newtheorem{defititle}[definition]{\Title}
\newtheorem{notation}[definition]{Notation}

\newtheorem{remark}[definition]{Remark}
\newtheorem{remarks}[definition]{Remarks}
\newtheorem{ex}[definition]{Example}
}
\newtheorem{prop}[definition]{Proposition}
\newtheorem{proposition-definition}[definition]{Proposition-Definition}
\newtheorem{lemma}[definition]{Lemma}
\newtheorem{thm}[definition]{Theorem}
\newtheorem{cor}[definition]{Corollary}

\newtheorem*{prop*}{Proposition}
\newtheorem*{theorem*}{Theorem}

\newtheorem{innercustomthm}{Theorem}
\newenvironment{customthm}[1]
  {\renewcommand\theinnercustomthm{#1}\innercustomthm}
  {\endinnercustomthm}

\newtheorem{innercustomdef}{Definition}
\newenvironment{customdef}[1]
  {\renewcommand\theinnercustomdef{#1}\innercustomdef}
  {\endinnercustomdef}

\newcommand{\cO}{\mathcal{O}}
\newcommand{\cF}{\mathcal{F}}

\newcommand{\id}{{\hbox{id}}}

\newcommand{\cf}{{\it cf.}\/ }

\newcommand{\vX}{\mathfrak{X}}

\def\gpd{\,\lower1pt\hbox{$\longrightarrow$}\hskip-.24in\raise2pt
             \hbox{$\longrightarrow$}\,}

\renewcommand{\latticebody}{\drop@{ }}

\newcommand{\Z}{\ensuremath{\mathbb Z}}

\newcommand{\R}{\ensuremath{\mathbb R}}

\newcommand{\T}{\ensuremath{\mathbb{T}}}

\newcommand{\g}{\ensuremath{\mathfrak{g}}}

\newcommand{\A}{\ensuremath{\mathfrak{a}}}
\newcommand{\h}{\ensuremath{\mathfrak{h}}}

\newcommand{\cL}{{\mathcal L}}
\newcommand{\cA}{\mathcal{A}}

\newcommand{\ZZ}{\ensuremath{\mathbb Z}}
\newcommand{\CC}{\ensuremath{\mathbb C}}
\newcommand{\RR}{\ensuremath{\mathbb R}}
\newcommand{\PP}{\ensuremath{\mathbb P}}

\newcommand{\bt}{\mathbf{t}}                  
\newcommand{\bs}{\mathbf{s}}                  

\newcommand{\fls}{log-f symplectic}

\def\act{\mathbin{\hbox{$<\kern-.4em\mapstochar\kern.4em$}}}
\def\ract{\mathbin{\hbox{$\mapstochar\kern-.3em>$}}}

\def\PB(#1,#2,#3,#4){\left\{\begin{matrix}#1&\!\!\!\stackrel{?}{\longrightarrow}&\!\!\!#2\\
\downarrow&&\!\!\!\downarrow\\
#3&\!\!\!\stackrel{?}{\longrightarrow}&\!\!\!#4\end{matrix}\right\}}

\def\pb(#1,#2,#3,#4){ \hom(#1 \to #3, #2 \to #4)}

\begin{document}

\begin{center}
{\Large\bf Almost regular Poisson manifolds and their holonomy groupoids

\bigskip

{\sc by Iakovos Androulidakis and Marco Zambon}
}
 
\end{center}

{\footnotesize
National and Kapodistrian University of Athens
\vskip -4pt Department of Mathematics
\vskip -4pt Panepistimiopolis
\vskip -4pt GR-15784 Athens, Greece
\vskip -4pt e-mail: \texttt{iandroul@math.uoa.gr}
  
\vskip 2pt KU Leuven
\vskip-4pt  Department of Mathematics
\vskip-4pt Celestijnenlaan 200B box 2400
\vskip-4pt BE-3001 Leuven, Belgium.
\vskip-4pt e-mail: \texttt{marco.zambon@kuleuven.be}
}
\bigskip
\everymath={\displaystyle}

\date{today}

\begin{abstract}\noindent
We look at Poisson geometry taking the viewpoint of singular foliations, understood as suitable submodules generated by Hamiltonian vector fields rather than partitions into (symplectic) leaves. The class of Poisson structures which behave
best from this point of view, are those whose submodule generated by Hamiltonian vector fields arises from a \emph{smooth} holonomy groupoid. We call them \emph{almost regular Poisson structures} and determine them completely. They include regular Poisson and log symplectic manifolds, as well as several other Poisson structures whose symplectic foliation presents singularities.

We show that the holonomy groupoid associated with an almost regular Poisson structure is a Poisson groupoid, integrating a naturally associated Lie bialgebroid. The Poisson structure on the holonomy groupoid is regular, and as such it provides a desingularization.  
The holonomy groupoid is ``minimal'' among Lie groupoids which give rise to the submodule generated by Hamiltonian vector fields. This implies that, in the case of log-symplectic manifolds, the holonomy groupoid coincides with the symplectic groupoid constructed by Gualtieri and Li.
 Last, we focus on the integrability of almost regular Poisson manifolds and exhibit the role of the second homotopy group of the source-fibers of the holonomy groupoid.
\end{abstract}

\setcounter{tocdepth}{2} 
\tableofcontents

\section*{Introduction}
\addcontentsline{toc}{section}{Introduction}


A great deal of the geometry, as well as the topology of a Poisson manifold $(M,\pi)$ is determined by the nature of its symplectic foliation, namely the partition of $M$ to immersed submanifolds (leaves) which carry a symplectic structure. Most of the symplectic foliations occurring in Poisson geometry are singular, that is to say, the dimension of the symplectic leaves varies.

In this paper we take an approach to symplectic foliations which is different from the traditional one,  adopting the Stefan-Sussmann definition for singular foliations used in \cite{AndrSk}. More precisely: instead of  considering the partition of $M$ into symplectic leaves, we focus on the $C^{\infty}(M)$-module of vector fields
\begin{equation*}
\cF := \sharp(\Omega^1_c(M)),
\end{equation*}
where $\sharp : T^{\ast}M \to TM$ is the map defined by contraction with $\pi$ and $\Omega^1_c(M)$ denotes the compactly supported sections of $T^{\ast}M$. In other words, $\cF$ is the $C^{\infty}(M)$-submodule of $\vX(M)$ generated by the compactly supported Hamiltonian vector fields.
Clearly the symplectic foliation can be recovered from $\cF$, by {integrating} the vector fields in $\cF$ (\cf \cite{Stefan, Sussmann}).

{The viewpoint of (Stefan-Sussmann) singular foliations provides us some new tools, see \S \ref{sec:proj}.
The most important of these tools is the holonomy groupoid $H(\cF)$ of the singular foliation $\cF$ \cite{AndrSk}, because:}
{
\begin{itemize}
\item[1)]  When $H(\cF)$ is a Lie groupoid, it induces the singular foliation $\cF$ on $M$ via the target map.
Hence $H(\cF)$  ``desingularizes'' the singular foliation on $M$:  the leaves of $\cF$ are obtained applying the target map to the source-fibers of $H(\cF)$.
\item[2)]  $H(\cF)$ is an adjoint groupoid, that is, it is the terminal object among all Lie groupoids whose Lie algebroid is isomorphic to that of $H(\cF)$. 
\end{itemize}
}
{For this reason we restrict our attention to those Poisson manifolds for which $H(\cF)$ is a Lie groupoid. In that case it turns out that the Poisson structure $\pi$ on $M$ is also ``desingularized'' by $H(\cF)$. We now elaborate on this.}
 
\paragraph{\underline{{Almost regular Poisson manifolds}}}

{The natural regularity condition from the  singular foliations point of view leads to the following definition, which we present in \S \ref{sec:aregp} (Def. \ref{def:almostregnew}):

\begin{customdef}{A}\label{def:almostregnewIntro}
A Poisson manifold $(M,\pi)$ is \emph{almost regular} if there is an isomorphism of $C^{\infty}(M)$-modules $\cF\cong \Gamma_c(A)$ for some vector bundle $A$ over $M$.
\end{customdef}
 
Almost regular Poisson structures include these two classes (the intersection of which are the symplectic manifolds, i.e. the full-rank Poisson tensors):

\begin{itemize}
\item \emph{Regular}; that is Poisson structures such that $\pi$ has constant rank.
\item \emph{Poisson structures having full rank on a dense open subset}. These include \emph{Log-symplectic} structures, i.e. Poisson structures on manifolds $M$ -- necessarily even dimensional manifolds, say of dimension $2n$ -- such that $\wedge^n\pi$ is a section of  the line bundle $\wedge^{2n}TM$ which is transverse to the zero section. 
They also include scattering symplectic structures \cite{LaniusScattering}.
\end{itemize}

The following   characterization of almost regular Poisson manifolds (see Theorem \ref{thm:aregpois})  is entirely in terms of the leaves of the symplectic foliation rather than the module $\cF$:
 \begin{customthm}{B}\label{thm:introchar}
A Poisson manifold $(M,\pi)$ is almost regular if and only if:
\begin{enumerate}
\item[i)] The subset $M_{\mathrm{reg}}$ where $\sharp$ has maximal rank is  dense in $M$, and
\item[ii)] There is a  (constant rank) distribution $D$ of $M$, such that at every $x \in M_{\mathrm{reg}}$ the symplectic leaf of $(M,\pi)$ at $x$ is an open subset of the leaf of $D$ at $x$.
\end{enumerate}

In this case: the distribution $D$ is unique, it is integrable, and 
its leaves are Poisson submanifolds.
\end{customthm}

The simplest example of an almost regular Poisson manifold that does not fall into the above two extreme classes is probably   $(\RR^3,t\partial_{x}\wedge \partial_{y})$, the dual of the Lie algebra of the Heisenberg group. In \S \ref{sec:examples} we explain how to construct almost regular Poisson structures in several ways, such as: from other almost regular ones by multiplying with suitable Casimir functions; via classical constructions in foliation theory; by taking a Poisson manifold equipped with a foliation by cosymplectic submanifolds, obtaining a foliated version of log-symplectic structures.

 \paragraph{\underline{{Desingularizing almost regular Poisson structures}}}
  
{
The principal motivation for introducing  Definition \ref{def:almostregnewIntro} is the following: for any Poisson manifold $(M,\pi)$, denoting by  $\cF$ the  $C^{\infty}(M)$-module defined at the beginning of this introduction, 
\begin{center}\emph{$H(\cF)$ is a Lie groupoid if{f} $(M,\pi)$ is an almost regular Poisson manifold}.
  \end{center}
This result is mainly due to the work of Debord \cite{DebordJDG} (for a more detailed account of the proof, see Remark \ref{rmk:smoothholgpd}). Now let  $(M,\pi)$ be almost regular. The vector bundle appearing in Def. \ref{def:almostregnewIntro}, which is just $D^*$, acquires canonically a Lie algebroid structure and is exactly the Lie algebroid of $H(\cF)$. 
In  \S \ref{sec:Poisgpd}  we prove:
\begin{itemize}
\item[3)]    $H(\cF)$ is a Poisson groupoid which induces the Poisson structure $\pi$ on $M$ via the target map. Further the Poisson structure $\Pi$ on $H(\cF)$ is regular. Hence $(H(\cF),\Pi)$  ``desingularizes'' the Poisson structure on $M$.
\end{itemize}
From 1), 2), 3) it follows that \emph{any} Lie groupoid integrating the Lie algebroid $D^{\ast}$ is a Poisson groupoid that ``desingularizes'' both the singular foliation $\cF$ and the Poisson structure $\pi$ on $M$.
}

{To put the above result into context, notice that 
the pair $(D^{\ast},D)$ is a Lie bialgebroid, as 
one checks using Thm. \ref{thm:introchar}.   It is trivial and well-known that the Lie algebroid $D$ admits an adjoint groupoid, namely the holonomy groupoid $H(D)$ of the foliation tangent to $D$. Note that $H(D)$ is a Poisson groupoid inducing   the Poisson structure $-\pi$ on $M$ via the target map, but the Poisson structure on $H(D)$ is not almost regular.}
 
 The results  we obtain in Prop. \ref{prop:HD}, Thm. \ref{thm:Poisgr} and Prop. \ref{prop:hfhd} can be summarized as follows:
 \begin{customthm}{C}\label{thm:introduality}
\begin{enumerate}
\item[i)] $H(D)$ and $H(\cF)$ are Poisson groupoids in duality, integrating respectively the Lie bialgebroids $(D,D^{\ast})$ and $(D^{\ast},D)$.
\item[ii)] There is a morphism of Lie groupoids $H(\cF) \to H(D)$ which is an anti-Poisson map and integrates the Lie algebroid morphism  
$\sharp \colon D^*\to D$.
\end{enumerate}
\end{customthm}

In the regular case we obtain nothing new, and in the Log-symplectic case we have $D=TM$ and the latter map is $(\bt,\bs) : H(\cF) \to M \times M$. 
 
\paragraph{\underline{Log-symplectic manifolds.}}

The fact that $H(\cF)$ is the adjoint groupoid of $D^{\ast}$ encompasses the case of Log-symplectic manifolds and their integrability. Recall that Gualtieri and Li in \cite{GuLi} constructed a symplectic groupoid for Log-symplectic manifolds using a blow-up construction (familiar by the work of Melrose on manifolds with boundary), and show that this groupoid is ``adjoint''.

Indeed, the Lie groupoid constructed by Gualtieri and Li is a holonomy groupoid, and therefore can be obtained not only by a blow-up procedure but also by the general and systematic construction of holonomy groupoids given in \cite{AndrSk}.

The interesting case of Log-symplectic manifolds  is a further motivation to cast our studies into singular foliation theory. This means that we enlarge the cases we examine to  the widest class of Poisson structures for which the holonomy groupoid is a Lie groupoid, i.e. -- as already mentioned -- to almost regular Poisson structures.
  
\paragraph{\underline{Integrability.}} 
Another point revealed by considering the module $\cF$, is that the integrability of an almost regular Poisson manifold $(M,\pi)$ depends {only} on the
holonomy of the symplectic foliation. 
We show in \S \ref{sec:integralgds} that if $(M,\pi)$ is almost regular  and the source-fibers of $H(\cF)$ are 2-connected, then $(M,\pi)$ is integrable. 

\paragraph{\underline{A hierarchy for Poisson structures.}}
In this paper we focus on Poisson manifolds whose underlying singular foliation $\cF$ is projective, i.e. there is an isomorphism $ \Gamma_c(A)\overset{\sim}{\to}\cF$ for some vector bundle $A$. Let us point out  an example of Poisson structure which does not belong to the ``almost regular'' class we discuss in this paper, but still satisfy a kind of projectivity, in fact at a higher level. 
Consider the Lie-Poisson manifold $M=\mathfrak{su(2)}^*$, the dual of the Lie algebra of $SU(2)$.
 In \S \ref{subsec:lin} we show that in this case, while the module $\cF$ is not projective,  the sections of $A_0: = T^{\ast}M$ whose images by the anchor map vanish form a projective  module of rank one. Whence, in this case $\cF$ has a projective resolution of length 1: 
 $$0\to \Gamma(A_1)\to \Gamma(A_0)\to \cF\to 0$$
 for $A_1$ the trivial line bundle.
 As far as brackets as concerned, recent work of Laurent-Gengoux, Lavau and Strobl  \cite{LLG}
 proves that given an arbitrary singular foliation $\cF$, every projective resolution of $\cF$ has an $L_{\infty}$-algebroid structure, and that the resulting $L_{\infty}$-algebroid is unique up to quasi-isomorphism.

Thus it seems that the projective dimension of the symplectic foliation $\cF$ might be used to give a hierarchy of Poisson structures.
Although such a hierarchy is formulated in terms of homological algebra, in essence it distinguishes Poisson structures to classes according to the nature of the singularities of their symplectic foliations. From this point of view, almost regular Poisson structures are 
the class of Poisson structures whose symplectic foliation has zero projective dimension. 
Determining and characterizing the classes with higher projective dimension in the fashion of the current paper, is a problem of different order, which we hope
will be addressed in the future.

\paragraph{Conventions:} All Lie groupoids are {\bf assumed to be source connected}, not necessarily Hausdorff, but with Hausdorff source-fibers. Given a Lie groupoid $G\rightrightarrows M$, we denote by $\bt$ and $\bs$ its target and source maps, and by $i\colon G\to G$ the inversion map. We denote by $1_x\in G$ the identity
element corresponding to a point $x\in M$, and by $1_M\subset G$ the submanifold of identity elements.
Two elements $g,h\in G$ are composable if $\bs(g)=\bt(h)$. We identify the Lie algebroid of $G$, which we sometimes denote by $Lie(G)$, with $\ker (d\bs)|_M$.

\paragraph{Acknowledgements:}
This work was partially supported by Marie Curie Career Integration Grant PCI09-GA-2011-290823 (Athens), by 
Pesquisador Visitante Especial grant 88881.030367/2013-01 (CAPES/Brazil), by IAP Dygest (Belgium), 
the long term structural funding -- Methusalem grant of the Flemish Government, and by SCHI 525/12-1 (DFG/G\"{o}ttingen).  {We thank  M. Gualtieri,  Y. Kosmann-Schwarzbach and S. Lavau for useful suggestions. We also thank O. Brahic and E. Hawkins for suggesting several improvements based on a  draft of this note.
}


\section{Projective singular foliations}\label{sec:proj}

{In this section we introduce projective singular foliations, recall that they have an associated Lie groupoid, and treat in detail the case of singular foliations arising from Lie algebroids. The connections to Poisson geometry are deferred to the  section \ref{sec:aregp}.}

\subsection{Projective singular foliations}\label{subsec:locp}

{After recalling briefly the notion of singular foliation \cite{AndrSk} (understood as a suitable submodule of vector fields), we introduce projective singular foliations.}
 
 \begin{definition}
 A {\em{singular foliation}} on $M$ is a $C^{\infty}(M)$-submodule $\cF$ of  $\vX_c(M)$ (the compactly supported vector fields), which is locally finitely generated and stable by the Lie bracket.
\end{definition}

By the work of  Stefan \cite{Stefan} and Sussmann \cite{Sussmann}, $\cF$ induces a partition of $M$ into injectively immersed connected submanifolds (called leaves), whose tangent spaces are given exactly by the evaluations of the elements of $\cF$. 

 \begin{notation}\label{not:algdcF}
Let $(M,\cF)$ be a singular foliation. We establish the following notation, which will be used throughout this sequel. Let $x\in M$.
\begin{itemize}
\item Let $L_x$ be the leaf of the foliation $(M,\cF)$ at $x$. We denote by $F_x$ the vector space $T_x L_x$.
\item The quotient
$$A^{\cF}_x:=\cF/I_x\cF$$
is a finite dimensional vector space (see \cite{AndrSk}). 
Here $I_x$ stands for ideal of functions in  $C^{\infty}(M)$ which vanish at $x$.
  
\item Recall from \cite[\S 1]{AndrSk} that  the evaluation map gives rise to a short exact sequence of vector spaces $$0 \to \A^{\cF}_{x} \to A^{\cF}_{x} \stackrel{ev_x}{\longrightarrow} F_x \to 0$$ Its kernel $\A^{\cF}_x$ becomes a Lie algebra by inheriting the Lie bracket of $\cF$. It is called the isotropy of the foliation $\cF$ at $x$.
\item Let $L_x$ be the leaf of $(M,\cF)$ at $x \in M$. We denote $A^{\cF}_{L_x} := \bigcup_{y \in L_x}A^{\cF}_{y}$. This is a transitive Lie algebroid {\cite{AZ1}}, namely an extension $$0 \to \A^{\cF}_{L_x} \to A^{\cF}_{L_x} \stackrel{ev}{\longrightarrow} TL_x \to 0$$ where $\A^{\cF}_{L_x} = \bigcup_{y \in L_x}\A^{\cF}_{y}$ is a bundle of Lie algebras. 
\end{itemize}
\end{notation}

The terminology in the following definition is motivated by the Serre-Swan theorem
\cite[Thm. 2]{Swan}:
\begin{definition}\label{def:projsfol}
A singular foliation $\cF$ on $M$ is called \emph{projective}  if there exist a vector bundle $A\to M$ and an isomorphism of $C^{\infty}(M)$-modules $\cF\cong\Gamma_{c}(A)$.
\end{definition}

\begin{lemma}
Given a projective singular foliation $\cF$:
\begin{enumerate} 
\item[i)]    The vector bundle  $A$ is a Lie algebroid, with Lie bracket obtained from the Lie bracket of vector fields via $\Gamma_{c}(A)\cong \cF$ and with anchor $\rho \colon A\to TM$ given by evaluation. Moreover $A$ is an \emph{almost injective} Lie algebroid, i.e. the anchor $\rho$ is injective\footnote{Equivalently,  
 the map at the level of sections $\Gamma(A)\to \Gamma(TM)$
is injective.} over an open dense subset   of $M$. 
\item[ii)]  Up to isomorphism covering $\id_M$ there is a unique  almost injective Lie algebroid $A$  so that $\rho(\Gamma_{c}(A))=\cF$. 
We refer to $A$ as an \emph{almost injective Lie algebroid inducing $\cF$}.
\end{enumerate}
\end{lemma}
\begin{proof} Item i) is well-known, so we prove only the 
 uniqueness statement ii). Let $A'$ be an another almost injective Lie algebroid such that $\rho'(\Gamma_{c}(A'))=\cF$. Then the composition 
$\phi\colon \Gamma_{c}(A)\overset{\rho}{\cong}\cF\overset{\rho'}{\cong}\Gamma_{c}(A')$ is an isomorphism of $C^{\infty}(M)$-modules. It induces an isomorphism of Lie algebroids $A\to A'$ which on the fiber over any $x\in M$ reads $A_x=\Gamma_{c}(A)/I_x\Gamma_{c}(A)\overset{\phi}{\to}\Gamma_{c}(A')/I_x\Gamma_{c}(A')=A'_x$. Here  we denote by $I_x$ the ideal of smooth functions on $M$ vanishing at $x$. 
\end{proof}

A leaf of a singular foliation is said to be \emph{regular} 
if it has  an open neighbourhood such that all leaves intersecting this  neighbourhood   have the same dimension. The set of regular points of $\cF$, which is by definition the union of the regular leaves, is denoted by $M_{\mathrm{reg}}$. It is an open dense subset of $M$ \cite[Prop. 1.5 c)]{AndrSk}.

\begin{remark}\label{rem:dense} If $\cF$ is a    projective foliation,   
the anchor $\rho$ of a corresponding almost injective Lie algebroid $A$ is injective exactly at points of $M_{\mathrm{reg}}$. It follows 
that the regular leaves are exactly the leaves of maximal dimension.
In particular, the restriction $\cF|_{M_{\mathrm{reg}}}$ to the  open dense subset $M_{\mathrm{reg}}$ is a regular foliation by leaves of  dimension $\mathrm{rk}(A)$. 
\end{remark}
 
Recall from Notation \ref{not:algdcF} the fiber $A^{\cF}_{x} = \cF/I_x\cF$ of the module $\cF$ at a point $x \in M$. The following lemma will be quite important in this sequel.

\begin{lemma}\label{lem:algcF}
\begin{enumerate}
\item[i)] The module $\cF$ is projective if{f} $\mathrm{dim}(A^{\cF}_{x})$ is a constant function of $x\in M$. 
\item[ii)] In this case there is a canonical Lie algebroid structure on 
$$A^{\cF} := \bigcup_{x\in M} A^{\cF}_{x}$$
 making it an almost injective Lie algebroid inducing $\cF$.
\end{enumerate}
\end{lemma}
\begin{proof} {We first prove i).}  
If $\cF$ is   projective we have $\cF\cong \Gamma_{c}(A)$ for a vector bundle $A$, so $A^{\cF}_{x}=\Gamma_{c}(A)/I_x\Gamma_{c}(A)=A_x$ (the fiber of $A$ at $x$) has the same dimension for all $x\in M$.

The proof of the converse implication is standard \cite{Swan}, we reproduce it for the reader's convenience. Fix $x\in M$, take a basis of $A^{\cF}_{x}$ and lift it to a set of generators $X_1,\dots,X_n$ of $\cF|_{U}$, where $U$ is a neighbourhood of $x$. {This is possible by \cite[Prop. 1.5]{AndrSk}.}
We claim that $$\sum_i f_iX_i\equiv 0\;\;\Rightarrow \;\; f_i= 0 \text{ for all }i,$$ where $f_i\in C^{\infty}(U)$. Indeed,  since the vector fields $X_1,\dots,X_n$ are generators of $\cF|_{U}$, for all $y\in U$
they induce a spanning set $[X_1]_y,\dots, [X_n]_y$ of $A^{\cF}_{y}$. This spanning set is actually a basis since $\dim(A^{\cF}_{y})=\dim(A^{\cF}_{x})$. Since the image of 
$\sum_i f_iX_i$ in the quotient $A^{\cF}_{y}$ is $\sum_i f_i(y)[X_i]_y$, we obtain $f_i(y)=0$ for all $i$, proving the claim.

We can cover $M$ by open subsets $U$   on which generators $X_1^U,\dots,X_n^U$ of $\cF|_U$
as above are defined. Consider the trivial rank $n$ vector bundle over $U$, with $X_1^U,\dots,X_n^U$ as canonical frame. By the above claim, on non-empty overlaps $U\cap V$ there are \emph{unique} functions $g^{UV}_{ij}\in C^{\infty}(U\cap V)$ such that
$X_i^U=\sum_jg_{ij}^{UV}X_j^V$ on $U\cap V$. The uniqueness implies that 
we can ``glue'' the trivial  vector bundles to a vector bundle, denote by $A^{\cF}\to M$,  using  the $g^{UV}_{ij}$  as transition functions.

The above claim also implies that the natural vector bundle map ${A^{\cF}}\to TM$ given by evaluation induces an isomorphism $\Gamma_{c}(A^{\cF})\cong \cF$ at the level of sections. {This terminates the proof of i). It also shows that}
 the Lie bracket of vector fields on $\cF$ can be pulled back to $A^{\cF}$, making $A^{\cF}$ into an almost injective Lie algebroid over $M$ inducing $\cF$ {and proving ii).}\end{proof}

\subsection{Holonomy groupoids of projective singular foliations}\label{subsec:holgr}

We exhibit a few facts about Lie groupoids associated to projective singular foliations, for motivational purposes and for later use. Lie groupoids will not be used until \S \ref{sec:Poisgpd}.

\begin{prop}\label{prop:adj}
Let $\cF$ be a projective singular foliation. Denote by $A$   the almost injective Lie algebroid inducing $\cF$. 
Then there exists an adjoint groupoid $G$ of $A$, unique up to isomorphism. This means:  
\begin{enumerate}
\item[i)] $G$ is a Lie groupoid integrating $A$ (in particular $A$ is an integrable Lie algebroid),
\item[ii)] any (source-connected) Lie groupoid integrating $A$ has a surjective morphism onto $G$ which differentiates to the identity at the Lie algebroid level. \end{enumerate}
\end{prop}
\begin{proof}
This is  \cite[Prop. 3.8]{AndrSk}, which summarizes result of various authors, among them Debord \cite[Thm. 1 and Thm. 2]{DebordJDG}. The first item also follows from \cite{CrFeLie}.
\end{proof}

Given  {any} singular foliation $\cF$, there is a canonically associated \emph{topological} groupoid $H(\cF)$, called \emph{holonomy groupoid} \cite{AndrSk}. 
 When $\cF$ is {projective}, by  \cite[Prop. 3.9]{AndrSk} there is a canonical  isomorphism
 $$H(\cF)\cong G$$ where $G$ is as in  Prop. \ref{prop:adj}.

\begin{remark}\label{rmk:smoothholgpd}
In particular,  if $\cF$ is {projective} then $H(\cF)$ is a Lie groupoid (integrating $A$). 
The reciprocal also holds. This provides the main motivation for considering projective singular foliations in this note: 

\begin{center}$H(\cF)$ is a Lie groupoid if{f} $\cF$ is projective.\end{center}
For the direction ``$\Rightarrow$'' of this equivalence, recall from \cite[Cor. 2.2]{Debord2013} that for any singular foliation, every $\bs$-fiber $H(\cF)_x$ is a smooth manifold and the fiber $\cF_x$ coincides with the tangent space of $H(\cF)_x$ at the identity. Therefore, when $H(\cF)$ is a Lie groupoid, the fiber $A^{\cF}_x$ has constant dimension as the basepoint $x$ runs through $M$. It follows from Lemma \ref{lem:algcF} that $\cF$ is a projective module.
\end{remark}

We now describe $G=H(\cF)$ explicitly in terms of any Lie groupoid $\Gamma$ integrating $A$. The following proposition follows from \cite[Ex. 3.4(4)]{AndrSk}. Here we provide a   direct proof, which relies only on Debord's work \cite{DebordJDG}. It makes use of the fact that $G$ is a quasi-graphoid.
Recall \cite[Def-Prop. 1]{DebordJDG} that a \emph{quasi-graphoid} is a Lie groupoid 
with the property that for any bisection $b$ defined on an open subset $U$, the condition\footnote{In \cite{AndrSk} a bisection $b$ satisfying
 $\bt\circ b=\id_U$ is said to carry the identity.}
 $\bt\circ b=\id_U$ implies that $b=\epsilon|_U$, where $\epsilon$  denotes the inclusion of the identities in the Lie groupoid.  Here we view bisections as sections of the source map.
\begin{prop}\label{lem:explicit}
Let $\Gamma$  be any Lie groupoid integrating $A$. Then
 $H(\cF)=\Gamma/\sim_{\Gamma}$
where 
$$g_1\sim_{\Gamma}g_2 
\Leftrightarrow \text{$\exists$ local bisections $b_1$ through $g_1$  and  $b_2$ through $g_2$ s.t. $\bt\circ b_1=\bt\circ b_2$.}
$$
\end{prop}
\begin{proof}
We first show that $\Gamma/\sim_{\Gamma}$ is a Lie groupoid integrating $A$.
Clearly $\Gamma/\sim_{\Gamma}=\Gamma/I$, where 
\begin{equation}\label{eq:I}
I=\{g\in \Gamma: \text{$\exists$ a local bisection $b\colon U\to \Gamma$ through $g$   s.t. $\bt\circ b=\id_U$}\}.
\end{equation}
It suffices (see for example  \cite[Thm. 1.20]{GuLi}) to prove the following claims: 
\begin{enumerate}
\item Set-theoretically, $I$ is a normal subgroupoid of $\Gamma$ lying in the union of the isotropy groups of $\Gamma$.
\item Topologically, $I$ is an embedded Lie subgroupoid of $\Gamma$
and it is $\bs$-discrete (i.e. the intersection of $I$ with any $\bs$-fiber is discrete).
\end{enumerate}

\paragraph{Proof of Claim a)}: $I$ is a \emph{set-theoretic subgroupoid} of $\Gamma$, since bisections can be composed.
It is clear that every element $g\in I$ lies in an isotropy group, since the existence of a bisection $b$ as in \eqref{eq:I} implies that $\bt(b(y))=y$ for all $y\in U$, in particular for $y=\bs(g)$. Further, $I$ is a \emph{normal} subgroupoid of $\Gamma$: given $g\in I_x$ and $\gamma\in \bs^{-1}(x)$, we have that $\gamma g \gamma^{-1}\in I$ by using the bisection $\beta b \beta^{-1}$, where $b$ is a bisection as in \eqref{eq:I} and $\beta$ is any bisection through $\gamma$.
 
\paragraph{Proof of Claim b)}: Since  $\Gamma$ and  $H(\cF)$ are Lie groupoids integrating the same Lie algebroid, there is an open neighbourhood $V$ of the identity section in $\Gamma$ which is isomorphic to an open neighbourhood   of the identity section in $H(\cF)$, and the latter is  a quasi-graphoid by \cite[Thm. 1, p. 492]{DebordJDG}. Since   $I$ equals the union of the images of all local bisections $b$ as in \eqref{eq:I}, this implies
  $I\cap V=\epsilon(M),$ where
 $\epsilon$  denotes the inclusion of the identities in $\Gamma$.
Let $g\in I$, and take a bisection $b\colon U\to \bs^{-1}(U)$ of $\Gamma$ as in \eqref{eq:I}.
Denote $r_{b}\colon \bs^{-1}(U) \to \bs^{-1}(U)$ the diffeomorphism given by right-multiplication by the bisection $b$. It maps $\epsilon(\bs(g))$ to $g$ and it preserves $I$, since
$b(U)$ lies in the subgroupoid $I$. Applying $r_b$ to
$$I\cap (V\cap \bs^{-1}(U))=\epsilon(U)$$
we obtain that the intersection of $I$ with $r_{b}(V)\cap \bs^{-1}(U)$ (an    open neighbourhood of $g$) is exactly $b(U)$. This shows both that $I$ is an embedded submanifold (hence, a \emph{Lie subgroupoid}) of $\Gamma$ and that $I$ is \emph{discrete}. 

This concludes the proof that $\Gamma/\sim_{\Gamma}$ is a Lie groupoid integrating $A$. Now, $\Gamma/I$ is a quasi-graphoid by construction, and it integrates $A$. The same holds for
 $H(\cF)$, by \cite[Thm. 1, p. 492]{DebordJDG}. The uniqueness statement \cite[Cor. 1, p. 475]{DebordJDG} implies that $\Gamma/I\cong H(\cF)$. 
\end{proof}

\subsection{Singular foliations arising from Lie algebroids}\label{sec:sfLA}

Until the end of \S \ref{sec:proj} we restrict our attention to singular foliations arising from a (not necessarily almost injective) Lie algebroid.
Let $E$ be any Lie algebroid and $\cF:=\rho(\Gamma_c(E))$ the corresponding singular foliation. {In this subsection we single out certain Lie algebras associated to $E$ (see the extension \eqref{eqn:essisotropyextn}).}

Fix $x \in M$. Denote    $F_x:=T_x L_x$ where $L_x$ is the leaf of $\cF$ through $x$. We display three short exact sequences of vector spaces.
\begin{enumerate}
\item Recall that the fiber $E_x$ fits in a short exact sequence of vector spaces
\begin{eqnarray}\label{eqn:isotropyLalgd}
0 \to \g_x \to E_x \stackrel{\rho_x}{\longrightarrow} F_x \to 0.
\end{eqnarray}
Its kernel $\g_x$ becomes a Lie algebra by inheriting the Lie bracket of $\Gamma_c (A)$. It is known as the isotropy Lie algebra of $E$ at $x$.
\item Recall from Notation \ref{not:algdcF} that the fiber $A^{\cF}_{x}$  fits in a short exact sequence of vector spaces:
\begin{eqnarray}\label{eqn:isotropyfol}
0 \to \A^{\cF}_{x} \to A^{\cF}_{x} \stackrel{ev_x}{\longrightarrow} F_x \to 0
\end{eqnarray}
\item The anchor map at the level of sections $\rho : \Gamma_c (E) \to \cF$ induces a short exact sequence of vector spaces, that relate the two previous ones:
\begin{eqnarray}\label{eqn:essisotropy}
0 \to \h_x \to E_x \stackrel{\widetilde{\rho}_x}{\longrightarrow} A^{\cF}_{x} \to 0.
\end{eqnarray}
Indeed $\rho$
sends the ideal $I_x\Gamma_c (E)$ onto the ideal $I_x\cF$, and $E_x=\Gamma_c (E) / I_x\Gamma_c (E)$. Explicitly, given $e\in E_x$ we have $\widetilde{\rho}_x(e)=
 \rho(\tilde{e}) \text{ mod }I_x\cF$,
where $\tilde{e}\in \Gamma(E)$ is any extension  of $e$.
\end{enumerate}

It turns out that   $\h_x$   plays a crucial role, so it is worth giving it a name.  We will justify this name in \S\ref{sec:integralgds}.

\begin{definition}\label{dfn:essisotropy} 
The kernel $\h_x$ of $\widetilde{\rho}_x\colon E_x \to A^{\cF}_{x}$ is called the \emph{germinal isotropy} of $E$ at $x \in M$. 
\end{definition}

\begin{lemma}\label{lem:nicechar} 
\begin{enumerate}
\item[i)] We have
$$\h_x=\{e\in E_x:\exists \text{ an extension }\tilde{e}\in \Gamma(E) \text{ s.t. } \rho(\tilde{e})\equiv 0\}.$$
\item[ii)]$ \h_x$ is a Lie subalgebra of $\g_x$.
\end{enumerate}
\end{lemma}
\begin{proof}
i) We prove only the inclusion ``$\subset$'', for the other one is trivial. Let $e\in \ker(\widetilde{\rho}_x)$, then for one (and therefore any) extension  $\tilde{e}\in \Gamma(E)$ we have $\rho(\tilde{e})\in I_x\cF=\rho(I_x\Gamma(E)))$. Therefore there is $\hat{e}\in I_x\Gamma(E)$ with 
$\rho(\tilde{e})=\rho(\hat{e})$. 
Since $\hat{e}$ vanishes at $x$,
$\tilde{e}-\hat{e}\in \Gamma(E)$ is also an extension of $e$, and maps identically to zero under $\rho$.

ii) Clearly $\h_x$ is contained in $\g_x=ker(\rho_x)$. {Item i)} makes it clear that $\h_x$ is closed under the Lie bracket of $\g_x$.
\end{proof}

 \begin{remark}
Lemma \ref{lem:nicechar}  i) shows that $x\mapsto \dim(\h_x)$ is a lower semicontinuous function on $M$. In particular, the dimension of $\h_x$ jumps in the opposite way as the dimension of the isotropy Lie algebras $\g_x$. 
\end{remark}
 
Now, it is quite easy to see that extensions \eqref{eqn:isotropyLalgd}, \eqref{eqn:isotropyfol} and \eqref{eqn:essisotropy} fit into the following commutative diagram of vector spaces and Lie algebras:
\[
\xymatrix{&0\ar[d]&0\ar[d]&&\\
& \h_x  
\ar@{=}[r] \ar[d] &  \h_x \ar[d]&& \\ 
0\ar[r] & \g_x \ar[r] \ar[d] & E_x \ar[r]^{\rho_x} \ar[d]_{\widetilde{\rho}_x} & F_x  \ar@{=}[d] \ar[r]&0 \\
0\ar[r] & \A^{\cF}_{x} \ar[d]\ar[r] & A^{\cF}_{x}  \ar[d]\ar[r]^{ev_x} & F_x \ar[r]&0\\
&0 &0 &&
}
\] 

In particular, the obvious diagram chasing shows that the kernels of the extensions \eqref{eqn:isotropyLalgd}, \eqref{eqn:isotropyfol} and \eqref{eqn:essisotropy} form an extension of Lie algebras themselves:
\begin{eqnarray}\label{eqn:essisotropyextn}
0 \to \h_x \to \g_x \stackrel{\widetilde{\rho}_x|_{\g_x}}{\longrightarrow} \A^{\cF}_{x} \to 0.
\end{eqnarray}

The germinal isotropy Lie algebra differs from the  isotropy  Lie algebra only at singular leaves.

\begin{prop}\label{prop:essisotrsing}
If the leaf $L_x$ at $x \in M$ is regular, then $\h_x=\g_x$. 
 \end{prop}
\begin{proof} 
We only need to prove $\g_x \subseteq \h_x$. 
The extension \eqref{eqn:isotropyLalgd} show that there is a neighbourhood $W$ of the leaf $L_x$ such that the dimension of the isotropy Lie algebras $\g_y$ is constant for every $y \in W$.
So any element $e\in\g_x$ can be extended to a section of $E$ lying in isotropy Lie algebras, therefore by definition $e$ lies in $\h_x$. An alternative  proof consists in 
noticing that $\A^{\cF}_{x}=0$ and applying the short exact sequence \eqref{eqn:essisotropyextn}.
\end{proof}

\begin{remark}\label{rem:cleanextn}
Let $L_x$ be the leaf of $\cF$ at $x \in M$. Recall that the restriction $E_{L_x}$ of $E$ to $L_x$ is a transitive Lie algebroid. Using   Notation \ref{not:algdcF} we see that it fits in an extension of vector bundles 
\begin{eqnarray}\label{eqn:cleanextn}
0 \to \h_{L_x} \to E_{L_x} \stackrel{\widetilde{\rho}}{\longrightarrow} A^{\cF}_{L_x} \to 0
\end{eqnarray}
where $\h_{L_x}$ is the bundle of Lie algebras $\h_{L_x}=\bigcup_{y \in L_x}\h_y$.
In the terminology of Brahic \cite{Brahic}, the fact that \eqref{eqn:essisotropyextn} is a short exact sequence says that \eqref{eqn:cleanextn} is a \emph{clean extension}. We will need this fact in \S \ref{subsec:extensions}.
\end{remark}

\subsection{Projective singular foliations arising from Lie algebroids}
\label{subsec:LAprojs}

{We specialize the setting of the previous subsection to projective singular foliations.}

This proposition follows immediately from Lemma \ref{lem:algcF} and Def. \ref{dfn:essisotropy}:
\begin{prop}\label{prop:projE} Let $E$ be any Lie algebroid
and  $\cF:=\rho(\Gamma_c(E))$ the associated singular foliation. Then 
$\cF$ is   projective if{f} $\dim(\h_x)$ is a constant function of $x\in M$.
\end{prop}

\begin{remark}\label{rem:alminj}
 Denote $\h=\bigcup_{x\in M} \h_x$, a bundle of Lie algebras. We have an isomorphism of Lie algebroids $$E/\h\cong A^{\cF},$$
where the latter was introduced in   Lemma \ref{lem:algcF} and is an   almost injective Lie algebroid inducing $\cF$.
To see this, notice that  the map $\tilde{\rho}\colon E\to A^{\cF}$  defined pointwise in 
eq. \eqref{eqn:essisotropy}
is a morphism of Lie algebroids with kernel $\h$.
 \end{remark}

We present an example of projective singular foliation arising from a Lie algebroid. A large class of examples, arising for   cotangent Lie algebroids of Poisson manifolds, will be presented in \S \ref{sec:examples}.

\begin{ex}\label{ex:fact}
Let $M$ be any manifold, $\mathfrak{k}$ a Lie algebra, and $\sigma\colon \mathfrak{k} \to \vX(M)$ a Lie algebra  homomorphism, i.e. an infinitesimal (right) action of the Lie algebra $\mathfrak{k}$ on $M$. Assume that the vector bundle map  $\rho \colon \mathfrak{k}\times M\to TM, (v,p)\mapsto 
\sigma_p(v)$ is fiber-wise injective\footnote{An extreme case is when the action is almost free.} on a dense subset of $M$. This means exactly that 
the  transformation Lie algebroid $\mathfrak{k}\times M$ is almost injective. Hence 
the singular foliation induced by the infinitesimal action, $\cF_{\mathrm{act}}:=\langle \sigma(v):v\in \mathfrak{k}\rangle$, is projective.

For instance, consider the action of $\CC^*$ on $\CC$ by multiplication, which has the origin as a fixed point and is free elsewhere. 
It is easy to see that the action induces the singular foliation $\cF$ generated by the Euler vector field $E:=x\partial_x+y\partial_y$ and by the ``rotation'' vector field $X:=x\partial_y-y\partial_x$. Similarly, consider the action of $\CC^*$ on $\CC^2$ given by $z\cdot(w_1,w_2)=(zw_1,z^kw_2)$ for some $k\in \ZZ$. The induced singular foliation is generated by $E_1+kE_2$ and $X_1+kX_2$, where the indices denote the copy of $\CC$ on which the vector fields are defined. Both are projective singular foliations.  
\end{ex}

\begin{remark}\label{rem:lafol}

\begin{itemize}
\item[a)]   Let $\cF$ be a  singular foliation arising as above from a Lie algebroid. As expected, the fact that $\cF$ is projective or not really depends on $\cF$, and not just on the underlying partition of the manifold $M$ into leaves.

 For example, take $M=\RR^2$. Consider the singular foliation $\cF_1$ induced by (the transformation Lie algebroid of)  the action of $\CC^*$ on $\CC$ by multiplication, and $\cF_2$ induced by the action of $GL(2,\RR)$ on $\CC=\RR^2$. Both singular foliations induce
the same decomposition into two leaves of $\RR^2$, namely the origin and $\RR^2-\{0\}$. But $\cF_1$ is projective (see  Ex. \ref{ex:fact}), while $\cF_2$ is not. To see the latter, recall from \cite[Prop. 1.4]{AndrSk} that the fiber $(\cF_2)_0 = \cF_2 / I_0 \cF_2$ is the Lie algebra of $GL(2,\RR)$, whence it has dimension $4$. However, at every point $x \neq 0$ of $\R^2$ the fiber $(\cF_2)_x$ is $T_x\RR^2$, whence it has dimension $2$. So the dimension of $\h_x$ is not constant, whence Prop. \ref{prop:projE} implies that $\cF_2$ is not projective. 

\item[b)] For  projective singular foliations $\cF$ arising from the cotangent Lie algebroid of a Poisson manifold, we will see in Thm. \ref{thm:aregpois} that there is a regular foliation containing the regular leaves of $\cF$ as open subsets. This is not the case for projective singular foliations arising from arbitrary Lie algebroids (take for instance the singular foliation on $\RR^2$  generated by the Euler vector field $x\partial_x+y\partial_y$).
\end{itemize}
\end{remark}

\section{Almost regular Poisson {structures}}\label{sec:aregp}

{At this stage we bring Poisson geometry into the picture. That is, we consider the singular foliation associated to a Poisson manifold (via the cotangent Lie algebroid) and focus on Poisson manifolds for which this  singular foliation is projective.}

\subsection{Singular foliations arising from Poisson structures}\label{subsec:pois}

In \S \ref{subsec:LAprojs} we considered projective singular foliations arising from
Lie algebroids. Now we specialize even further: we let $(M,\pi)$ be a Poisson manifold and  $$\cF:=\rho(\Gamma_c(T^*M))$$ the corresponding singular foliation, where $T^*M$ is the Lie algebroid of the Poisson manifold. The anchor $\rho \colon T^*M\to TM$ is given by contraction with $\pi$ and is therefore skew-symmetric, 
implying that $\ker(\rho_x)=(\mathrm{Im}(\rho_x))^{\circ}$ for all points $x$. So from Lemma \ref{lem:nicechar} i) we get the following useful characterization:
\begin{equation}\label{eq:nicecharPois}
\h_x =\{\xi\in T^*_xM:\exists \text{ a local extension to a $1$-form }\tilde{\xi} \text{ s.t. } \tilde{\xi}|_{TL}= 0 \text{ for any symplectic leaf $L$}\}.
\end{equation}
 
In fact, for Poisson manifolds we can specify  the position of $\h_x$ inside the Lie algebra $\g_x:=Ker(\rho_x)$, improving Lemma \ref{lem:nicechar} ii).

\begin{prop}\label{lem:center}
Let $(M,\pi)$ be a Poisson manifold, and consider the Lie algebra
$\g_x:=Ker(\rho_x)$ for $x\in M$. The germinal isotropy Lie algebra $\h_x$ lies in $Z(\g_x)$, the center of $\g_x$.
\end{prop}
\begin{proof}
Fix $x\in M$, a local 1-form $\alpha$ with $\rho\alpha\equiv 0$, and a local 1-form $\beta$ with $\beta_x\in \g_x$ (i.e. $\rho_x\beta_x=0$).  By definition of bracket on the Lie algebroid $T^*M$ 
we have $$[\alpha,\beta]=\cL_{\rho\alpha} \beta-\cL_{\rho\beta} \alpha-d\langle \rho\alpha,\beta \rangle=-\cL_{\rho\beta} \alpha
=-\iota_{\rho\beta}d\alpha-d\langle {\rho\beta},\alpha \rangle.
$$ where in the second equality two terms vanish because of $\rho\alpha\equiv 0$. 
This vanishes  at the point $x$  since $\rho_x\beta_x=0$ and
$-\langle {\rho\beta},\alpha \rangle=\langle {\rho\alpha},\beta \rangle\equiv 0$.
\end{proof}

\begin{remark}
\begin{itemize}
\item[a)]
 When $L_x$ is a regular leaf $\g_x$ is  abelian, so the above statement  is substantial only when $L_x$ is a singular leaf.
\item[b)] As expected,
$\h_x\neq  Z(\g_x)$ usually: when $M=\RR^2$ and $\pi=x^2\partial_x\wedge \partial_y$, we have $\h_0=0$, and the Lie algebra $\g_0$ is abelian (because $[dx,dy]=d\{x,y\}=2xdx$).
\end{itemize} 
\end{remark}

As an example, we compute germinal isotropies at the origin for linear Poisson manifolds.

\begin{prop}\label{prop:linPois}
Let   $(\g,[\;,\;]_{\g})$ be a finite dimensional Lie algebra. 
Consider the Poisson manifold $\g^*$, endowed with linear bivector field $\pi$ corresponding to the Lie-Poisson structure.

 At the origin $0\in \g^*$ one has $\h_0=Z(\g)$.
\end{prop}

\begin{proof}
Let $e\in T^*_0\g^*=\g$. Using the definition of $\h_0$ (see Def. \ref{dfn:essisotropy})
and the extension of $e$ to a constant 1-form,
we have $$e\in \h_0 \Leftrightarrow {\rho} (e) \in I_0\cF\Leftrightarrow {\rho} (e)\equiv0 \Leftrightarrow e\in Z(\g).$$
Here the second equivalence holds because the symplectic foliation $\cF$ is generated by linear vector fields, so $I_0\cF$ is generated by vector fields that vanish quadratically at the origin, while ${\rho}(e)$ is a linear vector field. The third equivalence holds because $({\rho}(e))(f)=[e,f]_{\g}$ for any $f\in C^{\infty}_{lin}(\g^*)=\g$.
\end{proof} 

\begin{remarks}[On linearizability]\label{rem:linearization}

\begin{enumerate}
\item Let $(M,\pi)$ be any Poisson manifold  and $x\in M$.  
Prop. \ref{prop:linPois} shows that for the linear Poisson structure on the dual of $\g_x:=Ker(\rho_x)$, at the origin we have $\h_0 = 
Z\g_x$. For points $x$ at which $\pi$ vanishes (so $\g_x=T^*_xM$), we deduce: if the Poisson structure $\pi$ is linearizable\footnote{This means that,  nearby $x$, $\pi$ is isomorphic to a neighbourhood of the origin in  the linear Poisson manifold $(\g_x)^*$.} at $x$, then necessarily $\h_x=Z\g_x$. Notice that the inclusion ``$\subset$''  always holds by Prop. \ref{lem:center}.
\item Conn's linearization theorem \cite{ConnSmooth} states that  if the isotropy Lie algebra $\g_x$ is a semisimple Lie algebra of compact type, then $\pi$ is linearizable nearby $x$. Item a) is consistent with Conn's theorem, since semisimple Lie algebras have trivial center, whence $\h_x=Z\g_x=\{0\}$.  
\end{enumerate}
\end{remarks}

\subsection{Almost regular Poisson manifolds}\label{sec:almregmfd}

{We  introduce}
the objects of interest of this note, namely Poisson structures whose symplectic foliation is projective:

\begin{definition}\label{def:almostregnew}
A Poisson manifold $(M,\pi)$ is \emph{almost regular} if $\cF:=\rho(\Gamma_c(T^*M))$ is a \emph{projective}  singular foliation (see Def. \ref{def:projsfol}).  
\end{definition}

By 
Prop. \ref{prop:projE} we obtain a first characterization:
\begin{prop}\label{prop:firstc}
A Poisson manifold $(M,\pi)$ is \emph{almost regular} if $\h_x$ has constant rank,
where $\h_x$ is given by eq. \eqref{eq:nicecharPois}.
\end{prop}

It is clear from eq. \eqref{eq:nicecharPois} that $\h$  depends only on the partition of $M$ into immersed leaves. In other words, given a manifold $M$ and Poisson bivector fields $\pi_1$ and $\pi_2$  which give the same partition of $M$ into leaves (but possibly with different symplectic forms on the leaves, and inducing different $C^{\infty}(M)$-modules of vector fields  $\cF_1$  and $\cF_2$):  $\cF_1$ is projective if{f} $\cF_2$ is projective. {Said otherwise:}

\begin{prop}
Given a Poisson manifold, the property of being  ``almost regular'' depends only on the partition of the manifold into immersed leaves.
\end{prop}

The above  is somewhat surprising, and it contrasts with the situation for general Lie algebroids described in Rem. \ref{rem:lafol}. 

We can make Prop. \ref{prop:firstc} more explicit:

\begin{thm}\label{thm:aregpois}Let  $(M,\pi)$ be a Poisson manifold, and denote\footnote{We denote this subset by $M_{\mathrm{reg}}$  since, for almost regular Poisson manifolds, it agrees with the set of regular points 
as defined before Rem. \ref{rem:dense}.}
 by $M_{\mathrm{reg}}$ the open subset of $M$ on which $\pi$ has maximal rank. 
$(M,\pi)$ is  
 {almost regular}   if{f}
  \begin{itemize}
\item[i)] $M_{\mathrm{reg}}$ is dense in $M$
\item[ii)] there is a
 distribution $D$ on $M$ such that  for all $x\in M_{\mathrm{reg}}$ one has $D_x=T_xL$, where $L$ denotes the symplectic leaf through $x$.
\end{itemize}
\end{thm}

\begin{remark} Given an almost regular Poisson manifold $(M,\pi)$:
\begin{enumerate}
\item the  distribution $D$ as above is unique.  Indeed, on the open dense subset $M_{\mathrm{reg}}$  the  distribution $D$ is clearly determined by $\pi$, and if two distributions agree on an open dense subset then they agree everywhere.  
 
\item $D$ is involutive. Indeed by ii)  the distribution $D$ is involutive on $M_{\mathrm{reg}}$, which is an open dense subset, and therefore {it is involutive} on the whole of $M$. 
\item the involutive distribution $D$ integrates to a {regular} foliation   by Poisson submanifolds. This is clear at points of $M_{\mathrm{reg}}$, and it is true on the whole of $M$ by a continuity argument.
\end{enumerate}
\end{remark}

\begin{proof}[Proof of Thm. \ref{thm:aregpois}] 
If   $x\in M_{\mathrm{reg}}$ then the subspaces $\ker(\rho_y)$
have constant dimension for all points $y$ 
nearby $x$,  so that  Lemma \ref{lem:nicechar} i) implies
$\h_x= \ker(\rho_x)$. Taking annihilators we see that 
$\h^{\circ}_x=T_xL$ for each $x\in M_{\mathrm{reg}}$, where $L$ is the symplectic leaf through $x$.

``$\Rightarrow$'' Assume that $\cF$ is projective. Condition i) is satisfied by Rem. \ref{rem:dense}. Further, {by Prop. \ref{prop:firstc},}
$\h$ is a vector subbundle of $T^*M$, so its annihilator $$D:=\h^{\circ}$$ is a distribution on $M$. As seen above, at regular points $x\in M_{\mathrm{reg}}$ we have $\h^{\circ}_x=T_xL$, where $L$ is the symplectic leaf through $x$. So condition ii) is satisfied.

  ``$\Leftarrow$'' Let $y\in M$.
  A covector $\xi\in T_y^*M$ can be extended locally to a 1-form annihilating the symplectic leaves of $(M,\pi)$ if{f} $\xi\in D_y^{\circ}$. This follows from the fact that 
  for all $x\in M_{\mathrm{reg}}$ we have $D_x=T_xL$, where $L$ denotes the symplectic leaf through $x$, and that $M_{\mathrm{reg}}$ is dense in $M$.
By eq. \eqref{eq:nicecharPois} this means that  $\h_y=D_y^{\circ}$, i.e. $\h=D^{\circ}$. $D$ being a distribution implies that $\h$ has constant rank, so we can apply Prop. 
\ref{prop:firstc}.
 \end{proof}

\begin{remark}\label{rem:Dstar}
By Rem. \ref{rem:alminj}, an almost injective Lie algebroid that gives rise to $\cF$ is 
$$A^{\cF}\cong T^*M/D^{\circ}=D^*,$$ i.e. the cotangent bundle to the leaves of the regular foliation.
\end{remark}

We characterize almost regular Poisson structures directly in terms of the bivector field $\pi$.
\begin{prop}\label{prop:aregtensor}
Let $(M,\pi)$ be a Poisson manifold, denote by $k$ the minimal integer such that $\wedge^k \pi$ is not identically zero and $\wedge^{k+1} \pi\equiv 0$.
 Denote $M_{\mathrm{reg}}:=\{x\in M: (\wedge^k \pi)_x\neq 0\}$. 
 
 $(M,\pi)$ is  
 {almost regular}   if{f}
  \begin{itemize}
\item[i)] $M_{\mathrm{reg}}$ is dense in $M$
\item[ii)] there is a rank 1 vector subbundle $K \subset \wedge^{2k}TM$ 
 such that $\wedge^k \pi\in \Gamma(K)$.
\end{itemize} 
\end{prop}
{Notice that the subbundle $K$ above, when it exists,   is unique.}
\begin{proof}
We make use of the characterization of  almost regular Poisson structures obtained in
Thm. \ref{thm:aregpois}. 
Clearly, at any point $x$ of $M$, $\pi$ has maximal rank (namely $2k$) if{f} $(\wedge^k \pi)_x\neq 0$.
Hence we only need to show that the existence of a subbundle $K$ as in $ii)$ is equivalent to the existence of a distribution $D$ as in Thm. \ref{thm:aregpois} $ii)$.

Given $D$, define $K:=\wedge^{2k}D$. It is a line bundle, and since the leaves of $D$ are Poisson submanifolds, we can view $\pi$ as a section of $\wedge^{2} D$, so $\wedge^k \pi\in \Gamma(K)$. Conversely, given $K$, notice that at points $x\in M_{\mathrm{reg}}$ we have $\RR(\wedge^k \pi)_x=K_x$. Hence $K_x$ is spanned by a decomposable element of $\wedge^{2k}T_xM$ {(namely, the wedge product of the elements of a basis of $T_xL$,  where $L$ denotes the symplectic leaf through $x$)}. The same holds for any arbitrary point of $M$, since being decomposable is a closed\footnote{This can be seen choosing coordinates to identify the tangents spaces of $M$ with $\RR^n$, and using the fact that the decomposable elements of the projective space $\PP(\wedge^{2k}\RR^n)$ form a closed subset because they are exactly the elements in the image of the Pl\"ucker embedding
 $$\mathrm{Grassmann}(2k, \RR^n)\to \PP(\wedge^{2k}\RR^n), W\mapsto [w_1\wedge\dots\wedge w_{2k}],$$
where $\{w_1,\dots,w_{2k}\}$ is any basis of the subspace $W\subset \RR^n$.}
 condition and $M_{\mathrm{reg}}$ is dense in $M$. Hence there exists a rank $2k$ distribution $D$ such that $\wedge^{2k}D=K$. At points $x\in M_{\mathrm{reg}}$ we have
$\RR(\wedge^k \pi)_x=K_x=(\wedge^{2k} D)_x$, implying that  $T_xL=D_x$, 
where $L$ denotes the symplectic leaf through $x$. 
\end{proof}

\section{Examples of almost regular Poisson structures}\label{sec:examples}
     
In \S \ref{subsec:extreme}-\S \ref{subsubsec:PD} we present several classes of examples of almost regular Poisson manifolds, mostly making use of Thm. \ref{thm:aregpois}. 
In \S \ref{subsec:lin} we will see that linear Poisson manifolds are seldom almost regular, and in \S \ref{subsec:actions} we present a natural construction involving Lie algebra actions  which however allows us to recover only regular Poisson structures.
 
\subsection{Two extreme cases of almost regular Poisson manifolds}\label{subsec:extreme}
There are two extreme classes of almost regular Poisson manifolds $(M,\pi)$, as can be seen from Thm. \ref{thm:aregpois}:
\begin{ex}\label{ex:simple}
\begin{enumerate}
\item The case $M_{\mathrm{reg}}=M$: one obtains exactly the regular Poisson manifolds, i.e. $\pi$ has constant rank.  In this case, the regular foliation integrating $D$ is the foliation  by symplectic leaves of $(M,\pi)$.

\item The case $D=TM$: one obtains exactly bivector fields $\pi$ which have full rank on a dense subset of $M$. A special case is given by log symplectic structures, also known as $b$-symplectic structures \cite{OriginalBman}, which we recall in \S \ref{subsubsec:PD}. Many examples and constructions for them are given in \cite{ExamGil} and \cite{SymplTopb}, both on orientable and non-orientable manifolds. A simple example  is the bivector field $\pi = y \partial_x\wedge \partial_y$ on $\R^2$; it has full rank everywhere except on the $x$-axis. Another special case is given by scattering symplectic structures \cite{LaniusScattering}.
\end{enumerate}
\end{ex}
  
\subsection{Multiplying by Casimir functions}  
  
The following is probably the simplest class of almost regular Poisson manifolds that does not 
belong to the above two extreme classes. It includes $(\RR^3,t\partial_{x}\wedge \partial_{y})$, the dual of the Heisenberg Lie algebra.
    
\begin{ex}\label{ex:simple2}
Let $N$ be a symplectic manifold, denote by $\pi_N$ the (full-rank) Poisson bivector on $N$ obtained by inverting the symplectic form. Then\footnote{Explicitly, the Poisson tensor is given by  $\pi(x)=t\cdot (\pi_N)(q)$ for all $x=(q,t)\in N\times \RR$.}
 $$(N\times \RR, \pi:=t\cdot\pi_N)$$
 is an almost regular Poisson structure, where $t$ denotes the coordinate on $\RR$. (It is a particular kind of Heisenberg-Poisson manifold.)
This it clear by Thm. \ref{thm:aregpois}. It can also be seen using Prop. \ref{prop:firstc}, for  eq. \eqref{eq:nicecharPois} makes it clear that for all $x\in N\times \RR$ we have that $\h_x$ is one-dimensional (it is spanned by $dt$). 
\end{ex}  
  
In general, one can use Casimir functions to produce new almost regular Poisson structures out of old ones. The proof of the following lemma is straightforward and is omitted.
\begin{lemma}
Let $(M,\pi)$ be an almost regular Poisson manifold. As earlier, we denote by $D$ the associated distribution.
\begin{enumerate}
\item A function $f\in C^{\infty}(M)$ is a Casimir function for $\pi$ if{f} $f$ is constant along the leaves of $D$.

\item Let $f$ be a Casimir function for $\pi$ with full support, i.e. $\{p\in M: f(p)\neq 0\}$ is dense in $M$. Then $(M,f\cdot \pi)$ is again an almost regular Poisson manifold. Its associated distribution is again $D$. 
\end{enumerate}
\end{lemma}

Ex. \ref{ex:simple2} can be recovered from the above lemma as follows:   
$(N\times \RR, \pi_N)$ is a regular Poisson manifold and $t\in C^{\infty}(N\times \RR)$ is a Casimir function with full support.
 
\subsection{Products and suspensions}

The class of almost regular Poisson structures is closed with respect to cartesian products, as can be seen easily from Thm. \ref{thm:aregpois}. More generally, one can perform a straightforward extension of a classical construction
from foliation theory, known as \emph{suspension}, as follows.
\begin{ex}
Let $(B,\pi_B)$ and $(N,\pi_N)$ be almost regular Poisson manifolds, and let $h$ be an action of the fundamental group $\pi_1(B,b)$ on $(N,\pi_N)$  by Poisson diffeomorphisms (here $b\in B$ is fixed). The Poisson structure $\pi_B$ can be lifted to the universal cover $\tilde{B}$ by the covering map, hence
 $\tilde{B}\times N$ is an almost regular Poisson manifold.
The discrete group $\pi_1(B,b)$ acts diagonally on it by Poisson diffeomorphisms and freely.
The quotient
$$(\tilde{B}\times N)/\pi_1(B,b)$$ is therefore an almost regular Poisson manifold, and comes with a Poisson map onto $B$.

For instance, let $\phi$ be a Poisson diffeomorphism of an almost regular Poisson manifold  $(N,\pi_N)$. $\phi$ generates an action of $\ZZ$ on $(N,\pi_N)$, and choosing $B=S^1$ we obtain an almost regular Poisson structure on the ``mapping torus'' $(\RR\times N)/\ZZ$, i.e. on  $([0,1]\times N)/((0,x)\sim (1,\phi(x))$. 
\end{ex}

\subsection{{\fls}  manifolds}\label{sec:fls}

Recall that a \emph{log symplectic} structure on a manifold $M^{2n}$ is a Poisson structure $\pi$ such that $\wedge^n\pi$ is a section of the line bundle $\wedge^{2n}TM$ that is transverse to the zero section \cite{OriginalBman}.  The zeros of this section then form a smooth hypersurface $Z$ (not necessarily connected), called the \emph{exceptional hypersurface}.

By analogy to this, Prop. \ref{prop:aregtensor} suggests to define:
\begin{definition}\label{def:folls}
Let $(M,\pi)$ be a   Poisson manifold {with $\pi\notequiv 0$}, denote by $k$ the unique integer such that $\wedge^k \pi$ is not identically zero and $\wedge^{k+1} \pi\equiv 0$.
We say that $(M,\pi)$ is  \emph{\fls} if there is a rank 1 vector subbundle $K \subset \wedge^{2k}TM$ 
 such that  $\wedge^k \pi$ {is a section of $K$ which is} transverse to the zero section.
\end{definition}
  
\begin{ex}
When the dimension of $M$ is even and $2k=dim(M)$, it follows from the definition that $\pi$ is a  log symplectic structure. This justifies in part the name ``{\fls}'', see Prop. \ref{prop:flogdense} below for a further justification.
\end{ex}
  
  {Let $(M,\pi)$ be a {\fls} manifold.}
  To simplify the notation let us denote $s:=\wedge^k \pi$ and $\mathrm{gr}(s):=\{s(x):x\in M\}\subset K$.
By definition, the transversality condition for $s$ at a point $x\in \mathrm{gr}(s)\cap M$ is $T_x\mathrm{gr}(s)+T_xM=T_xK$, which  is equivalent to $T_x\mathrm{gr}(s)\neq T_xM$ since $K$ has rank 1.
Notice that $$Z:=\mathrm{gr}(s)\cap M$$
is a codimension 1 submanifold (when it is not empty),
so in particular its complement is dense in $M$. By  Prop. \ref{prop:aregtensor} we conclude that $(M,\pi)$ is an almost regular Poisson structure. 

{In the rest of this subsection we address the following natural question.}
 Denote by $D$ the associated distribution as in Thm. \ref{thm:aregpois}. It satisfies
 $\wedge^{2k}D=K$, and  its leaves are Poisson submanifolds of $(M,\pi)$ of dimension $2k$. It is natural to wonder whether, for each leaf $P$ of $D$, the pair $(P,\pi_P)$ is a log symplectic manifold, where $\pi_P$ denotes the restriction of $\pi$ to $P$.
 The answer is ``no'', but we obtain a  positive partial result in Prop. \ref{prop:flogdense}.

\begin{lemma}\label{lem:notls}
Let $x\in Z\cap P$. Then $\wedge^k \pi_P$ is transverse to the zero section of $\wedge^{2k}TP$ at $x$ if{f} $D_x\not\subset T_xZ$.
\end{lemma}
\begin{proof}
{Using $T_xZ=T_x\mathrm{gr}(s)\cap T_xM$ we have}
$$D_x\subset T_xZ \Leftrightarrow D_x\subset T_x\mathrm{gr}(s|_P)\Leftrightarrow D_x= T_x\mathrm{gr}(s|_P).$$ On the other hand
the transversality condition for $\wedge^k (\pi_P)$ at $x$ is    is equivalent to 
$T_x\mathrm{gr}(\wedge^k (\pi_P))\neq T_xP$, i.e. to $T_x\mathrm{gr}(s|_P)\neq D_x$.
\end{proof}
 
Let us denote 
$$Z_{\mathrm{sing}}:=\{x\in Z:  D_x\subset T_xZ\}.$$
By Lemma \ref{lem:notls} this is the set of points $x$ of $Z$ where the leaf of $D$ through $x$  ``fails to 
be log symplectic at $x$''. Hence a leaf of $D$ is log symplectic
if{f} it does not intersect $Z_{\mathrm{sing}}$.
 
We present simple examples of {\fls} manifolds. More example are given
in the next subsection (see the text following Prop. \ref{prop:log} and Remark \ref{rem:flog}).

\begin{ex}\label{ex:r3}
Let $f\colon \RR^3\to \RR$ be a function such that $0$ is a regular value. 
 Consider the Poisson manifold
$(\RR^3,f\partial_{x_1}\wedge\partial_{x_2})$. Then
$\mathrm{gr}(\pi)$ is a section of the line bundle $K=\RR(\partial_{x_1}\wedge\partial_{x_2})$, and it is transverse to the zero section since $0$ is a regular value of $f$. Hence we have a  {\fls} manifold. Notice that by dimension reasons 
$Z_{\mathrm{sing}}=\{x\in f^{-1}(0): D_x= T_xZ \}$.
A consequence of this is that  any smooth curve with image in $Z_{\mathrm{sing}}$ lies in a single leaf of $D$.

Using $D_x=\mathrm{Span}\{\partial_{x_1},\partial_{x_2}\}$,
we see that: 
\begin{enumerate}
\item For $f=x_1$ we have  $Z=\{x_1=0\}$ and 
$T_xZ \neq D_x$ at all $x\in Z$, so $Z_{\mathrm{sing}}$ is empty. 
\item For $f=x_3$ we have  $Z=\{x_3=0\}$ and 
$T_xZ=D_x$ at all $x\in Z$, so $Z_{\mathrm{sing}}=Z$. 
(This is a special case of Ex. \ref{ex:simple2}).  
\item For $f=x_3-x_1^2$ we have $Z=\{x_3-x_1^2=0\}$ and $Z_{\mathrm{sing}}=\{x_1=0\}\cap Z$ is a line inside $Z$.
\end{enumerate}
\end{ex}

The above example is special also in that the rank of $D$ is exactly the dimension of $Z$. In general
the following proposition holds, stating that ``almost all'' leaves of $D$ are log symplectic manifolds, therefore justifying the name we gave in Def. \ref{def:folls}.
\begin{prop}\label{prop:flogdense}
Let $(M,\pi)$ be a {\fls} manifold.  Denote
 by $D$ the corresponding involutive distribution.
Assume that $M/D$, the quotient of $M$ by the foliation integrating $D$, has a\footnote{There is at most one smooth structure with this property.} smooth manifold structure so that the projection $\mathrm{pr}\colon M\to  M/D$ is a submersion. 

Then there is a measure zero  set $X\subset M/D$ such that for all $c\in (M/D)-X$,  the leaf $\mathrm{pr}^{-1}(c)$  is a log symplectic manifold.
\end{prop}
\begin{remark}\begin{itemize}
\item[a)] Of course $M/D$ is not always a smooth manifold. However one can always apply the above proposition locally, since for every point of $M$ there is a neighbourhood $U$ such that $U/(D|_U)$ is smooth. Indeed, as $U$ one can take the domain of a foliated chart for the foliation integrating $D$.
\item[b)]  Since $X$ has measure zero, the complement $(M/D)-X$ is a dense subset of $M/D$. (See \cite{LeeIntroSmooth} for more details). 
\end{itemize}
\end{remark} 
\begin{proof}
Let  $Z$ and  $Z_{\mathrm{sing}}$ be as above.
Consider the map $$\mathrm{pr}|_Z\colon Z\to M/D.$$
The set of critical points of $\mathrm{pr}|_Z$ is exactly $Z_{\mathrm{sing}}$. (This is seen noticing that at every $x\in Z$ the derivative ${(\mathrm{pr}|_Z)_*}_x$ has kernel $D_x\cap T_xZ$, and counting dimensions). Therefore the critical values of $\mathrm{pr}|_Z$ are exactly the points $c\in M/D$ such that  ${(\mathrm{pr}|_Z)^{-1}}(c)\cap Z_{\mathrm{sing}}\neq \emptyset$, i.e. exactly those for which $\mathrm{pr}^{-1}(c)$ is not log symplectic.
The classical Sard's theorem says that the set of critical values of any differentiable map between manifolds has zero measure.
\end{proof}
  
\subsection{Foliations by cosymplectic submanifolds}\label{subsubsec:PD}

Given a manifold $M$ with a Poisson structure $\pi$, an approach to construct a new Poisson structure on $M$ with some control of its symplectic foliation is the following: prescribe a foliation on $M$ so that the leaves of the foliation have an induced Poisson structure. In favorable cases the latter combine into a new Poisson structure on $M$. We take this approach to construct new almost regular Poisson manifolds {out of old ones, using a foliation by cosymplectic submanifolds. 
We start with a general statement (Lemma \ref{lem:pinew}), which we then specialize in order to  provide concrete  examples}.

\begin{remark}\label{rem:LA}
{Let $(P,\Pi)$ be a Poisson  vector space, i.e. $P$ is a vector space and $\Pi\in
\wedge^2P$. We denote by $\sharp\colon P^*\rightarrow P$  the map
induced by contraction with $\Pi$.} 

{
Let $W\subset P$ be a subspace.  
$W$ is called a \emph{cosymplectic}
subspace if $\sharp W^{\circ}\oplus W=P$. 
 In this case there is an induced bivector on $W$ (making $W$ a Poisson vector space), described as follows \cite{CrFePois}: its sharp map $\sharp_W \colon W^*\rightarrow W$
is given by \begin{eqnarray}\label{eq:inducedW}
\sharp_W  {\xi}=\sharp \tilde{\xi}
\end{eqnarray} where $\tilde{\xi}\in P^*$
is the unique extension of $\xi$ which annihilates $\sharp W^{\circ}$.}

{
For later reference, we also mention a few general facts about Poisson vector spaces  $(P,\Pi)$ . The image $\cO:=\sharp P^*$ is a symplectic vector space, and the symplectic form $\omega\in \wedge^2\cO^*$ is completely equivalent to $\Pi$.
For any subspace $Z\subset P$,
the symplectic orthogonal of $Z\cap \cO$ in
$(\cO,\omega)$ is $\sharp Z^{\circ}$. Hence $\sharp Z^{\circ}\cap Z$
 is the kernel of the restriction of $\omega$ to $Z\cap \cO$. Notice that when $Z$ is a cosymplectic subspace, this restriction is non-degenerate.} 
\end{remark}

{ 
Given a Poisson manifold $(M,\pi)$,
a \emph{cosymplectic
submanifold} $N$ is one whose tangent spaces are cosymplectic.  
By the above $N$ inherits a bivector field  $\pi_N$, which is a Poisson structure.
} 
\begin{lemma}\label{lem:pinew}
Let $(M,\pi)$ be a Poisson manifold so that $\pi$ has full rank (i.e. is symplectic) on a dense subset. Let $D$ be an involutive distribution on $M$ whose integral leaves are cosymplectic submanifolds.  

Consider the bivector field $\pi^{\mathrm{new}}$ on $M$ which, on each leaf $N$ of $D$, agrees with the Poisson bivector field $\pi_N$ induced by $\pi$.
Then $(M,\pi^{\mathrm{new}})$ is an almost regular Poisson structure. 
\end{lemma}
\begin{proof}
To show that $\pi^{\mathrm{new}}$ is a smooth bivector field on $M$ we argue that the corresponding sharp map $T^*M\to TM$ is smooth. {This holds since the extension $\tilde{\xi}$ of $\xi$ in eq. \eqref{eq:inducedW} is unique.}
Since the restriction of $\pi^{\mathrm{new}}$ to each leaf of $D$ is Poisson, it follows that $\pi^{\mathrm{new}}$ is a Poisson bivector field on $M$.

 Let $N$ be a leaf of the distribution $D$.    
The symplectic leaves of $(N,\pi_N)$ are the 
intersection of the symplectic leaves of $(M,\pi)$ with $N$ \cite[\S 9]{CrFePois}. {Denote by $M_{\mathrm{reg}}$ the dense subset of $M$ on which $\pi$ is symplectic.} At points of $M_{\mathrm{reg}}$, the tangent space of the symplectic leaf of $(N,\pi_N)$ is the tangent space to $TN$. In particular,
 the set of points where  $\pi^{\mathrm{new}}$ has maximal rank contains $M_{\mathrm{reg}}$, so it is dense.
It follows that $(M, \pi^{\mathrm{new}})$
is almost regular by Thm. \ref{thm:aregpois}.
\end{proof}

\begin{remark}\label{rem:open}
If a distribution $D$ on a Poisson manifold $(M,\pi)$ satisfies the cosymplectic condition $\sharp D_x^{\circ}\oplus D_x=T_xM$ at a point $x$, {by continuity}  it satisfies it in a neighbourhood of $x$ in $M$. 
\end{remark}
 Recall that on a log symplectic  manifold $M$ of dimension $2n$ with exceptional hypersurface $Z$, the symplectic leaves are given by  the connected components of $M-Z$, which are open, and the symplectic leaves lying in $Z$, which 
have dimension $2n-2$.

\begin{prop}\label{prop:log}
Let $(M,\pi)$ be a log symplectic  manifold. Let $D$ be an involutive distribution on $M$ such that
\begin{equation}\label{eq:cosyZ}
\sharp D_x^{\circ}\oplus D_x=T_xM
\end{equation}
 at all $x\in Z$. Then there exists a tubular neighbourhood $U$ of $Z$ in $M$ on which the Poisson bivector field $\pi^{\mathrm{new}}$ (defined as in Lemma \ref{lem:pinew}) is  
 almost regular.  
\end{prop}
\begin{proof}
  By Remark \ref{rem:open} there is a tubular neighbourhood $U$ of $Z$ in $M$ on which the decomposition  $\sharp D^{\circ}\oplus D=TM$ holds. Apply
  Lemma \ref{lem:pinew} to $(U,\pi|_U)$.
\end{proof}

\begin{remark}\label{rem:flog}
In Prop. \ref{prop:log}, $\pi^{\mathrm{new}}$ is actually a {\fls} structure on $U$.
Indeed $\pi$ decomposes as $\pi=\pi^{\mathrm{new}}+\pi^{\mathrm{trans}}$ where 
$\pi^{\mathrm{new}}\in \Gamma(\wedge^2D)$ and $\pi^{\mathrm{trans}}\in \Gamma (\wedge^2 \sharp D^{\circ})$ \cite[\S 9.1]{CrFePois}. So $\wedge^{n}\pi\cong\wedge^k \pi^{\mathrm{new}}\otimes \wedge^{n-k} \pi^{\mathrm{trans}}\in \Gamma(\wedge^{2k}D\otimes \wedge^{2(n-k)}\sharp D^{\circ})$
where $2n=\dim(M)$ and $2k=rank(D)$. Now $\wedge^{n-k} \pi^{\mathrm{trans}}$ is a nowhere vanishing section of the line bundle $\wedge^{2(n-k)}\sharp D^{\circ}$, hence the fact that $\wedge^{n}\pi$ is transverse to the zero section of $\wedge^{2n}TM$ implies that $\wedge^k \pi^{\mathrm{new}}$ is transverse to the zero section of $\wedge^{2k}D$.\end{remark}

We now describe a simple situation in which  Prop. \ref{prop:log} applies.
Let $(\cO,\omega)$ be a compact symplectic manifold. Let $\psi\colon \cO\to\cO$ be a symplectomorphism.
Endow
$$Z:=([0,1]\times \cO)/\;(0,x)\sim (1,\psi(x))$$
with the corank 1 Poisson structure $\pi_Z$ induced by the obvious product Poisson structure on $[0,1]\times \cO$. The   symplectic leaves of $\pi_Z$  are exactly the fibers of the canonical projection $Z\to S^1$, and are all symplectomorphic to $(\cO,\omega)$.
It is well-known that 
$$(Z\times \RR, \pi:=t\partial_t\wedge\partial_{\theta}+\pi_Z)$$
is a log symplectic manifold, where $\theta$ denotes   the coordinate on $[0,1]$ and $t$ the one on $\RR$. 

\begin{cor}\label{cor:log}
In the above setting, let  $\Delta$ be an involutive distribution on $\cO$ by symplectic subspaces, and assume that\footnote{I.e., assume that 
$\psi$ preserves the symplectic foliation integrating $\Delta$.}
    $\psi_*\Delta=\Delta$. Define the distribution $$D:= \RR\partial_t\oplus \RR \partial_{\theta}\oplus \Delta$$ on $Z\times \RR$. Then $D$ satisfies the assumptions of Prop. \ref{prop:log},
and therefore $\pi^{\mathrm{new}}$ is an almost regular Poisson structure on $Z\times (-\epsilon,\epsilon)$ for some $\epsilon>0$.  
\end{cor} 
\begin{proof}
We have to show that  $D$ satisfies the assumptions of Prop. \ref{prop:log}. The involutivity of 
$D$ follows from the following fact: the flow of $\partial_{\theta}$  on $Z$ preserves
the distribution $\Delta$.

{We show that eq. \eqref{eq:cosyZ} holds at all   $x\in Z$. 
By  Rem. \ref{rem:LA},
 $\sharp_xD_x^{\circ}$ is the symplectic orthogonal of $D_x\cap T_x\cO=\Delta_x$ in $T_x\cO$. So $\sharp_xD^{\circ}\oplus   \Delta_x= T_x\cO$, and the statement follows from this.}
 \end{proof}

We exhibit   examples for Cor. \ref{cor:log}. Notice that in that corollary the symplectic leaves of $\pi^{\mathrm{new}}$ at points of $Z$ are exactly the leaves of $\Delta$.
\begin{ex}
Consider the symplectic manifold $(\cO,\omega):=(\T^4,dx_1\wedge dx_2+dx_3\wedge dx_4)$. 

\begin{itemize}
\item[a)] 
 Take $\psi=\id_{\T^4}$. We have $Z\times (-\epsilon, \epsilon)=\T^5\times (-\epsilon, \epsilon)$, with log symplectic structure $\pi=t\partial_t\wedge\partial_{\theta}+\partial_{x_1}\wedge\partial_{x_2}+
\partial_{x_3}\wedge\partial_{x_4}.$ 

Take $$\Delta:=
\R(\sum_{i=1}^4a_i\partial_{x_i})\oplus \RR(\sum_{i=1}^4b_i\partial_{x_i})$$
where $a_1,\dots,b_4\in \R$ satisfy $(a_1b_2-a_2b_1)+(a_3b_4-a_4b_3)\neq 0$.
$\Delta$ is an involutive symplectic distribution on $\T^4$, and its leaves are all diffeomorphic to each other and may be diffeomorphic to a 2-torus, to a cylinder or to a plane.

In the simple case $\Delta:=
\R \partial_{x_1}\oplus \RR \partial_{x_2}$, on $\T^5\times (-\epsilon, \epsilon)$ we
 obtain the almost regular Poisson structure $\pi^{\mathrm{new}}=t\partial_t\wedge\partial_{\theta}+\partial_{x_1}\wedge\partial_{x_2}$. 

\item[b)] Fix $n,m\in \ZZ$. Let $\psi$ be the symplectomorphism of $\cO=\T^2\times \T^2$ which in matrix form is given by $
\left(\begin{array}{cc}1 & 0 \\n & 1\end{array}\right)\times
\left(\begin{array}{cc}1 & 0 \\m & 1\end{array}\right),
$  where we identify $\T^2$ with $\R^2/\Z^2$.(Notice that both matrices lie in $SL(2,\Z)$.)
$Z$ is the Whitney sum $X_n\oplus X_m$ of two $\T^2$-bundles over $S^1$, and the  6-dimensional manifold $Z\times {\RR}$ is log symplectic.

Let $\Delta:=
\R\partial_{x_1}\oplus \RR\partial_{x_2}$, and consider the corresponding Poisson structure $\pi^{\mathrm{new}}$ on $Z\times (-\epsilon, \epsilon)$.
The symplectic leaves of $\pi^{\mathrm{new}}$ lying in $Z$ are the 2-dimensional fibers of the $\T^2$-bundle 
$X_n\to S^1$, while the symplectic leaves on the complement of $Z$ are 4-dimensional and given by copies of  $X_n\times (0, \epsilon)$ and  $X_n\times (- \epsilon,0)$.

One obtains a more interesting symplectic foliation, when $n\neq 0$ and $n\neq -m$, by choosing the symplectic involutive distribution on $\cO$ to be
$$\Delta:=
\R(\partial_{x_1}+\partial_{x_3})\oplus \RR(\partial_{x_2}+\frac{m}{n}\partial_{x_4}).$$ Its leaves are again 2-tori, which however now ``wind around $\cO$''.
 \end{itemize}
\end{ex}

\subsection{Linear Poisson manifolds}\label{subsec:lin}
Linear Poisson structures (those given by duals of Lie algebras) are  {\emph{seldom}} almost regular, as we now discuss.

 \begin{ex}\label{ex:su2}
 The Poisson manifold $\mathfrak{su}(2)^*$, endowed with the canonical (linear) Poisson structure, is \emph{not} almost regular. The  symplectic foliation $\cF$ consists of all vector fields on $\mathfrak{su}(2)^*\cong \RR^3$ which are tangent to the leaves of $\cF$, which are exactly all concentric spheres about the origin and the origin itself. The Poisson structure is not
 almost regular
  since the regular foliation by concentric spheres on $\mathfrak{su}(2)^*-\{0\}$ can not be extended to a regular foliation on $\mathfrak{su}(2)^*$, so that the second condition in Thm. \ref{thm:aregpois} is not\footnote{The first condition is satisfied, since $M_{\mathrm{reg}}=\mathfrak{su(2)}^*-\{0\}$.} satisfied. Alternatively one can use Prop. \ref{prop:firstc} and notice that $\cF$ is not   projective by  eq. \eqref{eq:nicecharPois}: $dim(\h_p)=1$ for $p\in  \mathfrak{su(2)}^*-\{0\}$, but $\h_p=0$ at $p=0$, since every one form annihilating the concentric spheres must vanish at the origin. 
\end{ex}

\begin{remark}\label{rem:su2}
As mentioned in the introduciton, although the singular foliation $\cF$ associated to the Lie-Poisson manifold $ \mathfrak{su(2)}^*=:M\cong \RR^3$ is not projective,
it is not too far from being projective. The anchor map induces at the level of sections a surjective map $\sharp\colon \Gamma_c(T^*M)\to \cF$.
The kernel of this map is the  submodule 
generated by\footnote{Indeed, by applying the standard inner product in $\RR^3$, this statement is equivalent to saying that any smooth vector field on $\RR^3$ that is radial (i.e. perpendicular to the concentric spheres about the origin) is a $C^{\infty}(\RR^3)$-multiple of the Euler vector field $x\partial_x+
y\partial_y+z\partial_z$.}
$\alpha:=d({r^2}/{2})=r\cdot dr=x\cdot dx+y\cdot dy+z\cdot dz$,
hence it is a projective module. Denoting by $\RR\times M$ the trivial line bundle, we hence obtain an exact sequence of $C^{\infty}(M)$-modules \begin{equation}
\label{eq:sesmod}
0\to \Gamma_c(\RR\times M)\overset{\cdot \alpha}{\to} \Gamma_c(T^*M) \overset{\sharp}{\to} \cF\to 0,
\end{equation}
i.e. a projective resolution of $\cF$ of length 1.

For any $p\in M$, the relation to the germinal isotropy Lie algebras $\h_p$ is as follows. By quotienting every term of \eqref{eq:sesmod} with the ideal induced by $I_p$ (the functions vanishing at $p$), we obtain a sequence of vector spaces
$0\to \RR \overset{\cdot \alpha_p}{\to}  T_p^*M \overset{\sharp_p}{\to} \cF/I_p\cF \to 0.$ The image of the first map is exactly $\h_p$, by Lemma \ref{lem:nicechar} i). Notice that  this sequence fails to be exact at $\RR$ (i.e. the map $\cdot \alpha_p$ is not injective) for $p=0$.
\end{remark}
 
More generally we have the following.
\begin{prop}
Let $\g$ be a Lie algebra. $\g^*$ is an almost regular Poisson manifold if{f} the union of the symplectic leaves\footnote{As is well-known, the symplectic leaves of the Lie-Poisson structure are exactly the orbits for the coadjoint action of any connected Lie group integrating $\g$.} whose tangent spaces agree with the constant distribution $Z(\g)^{\circ}$ is dense in $\g^*$. Here $Z(\g)^{\circ}$ denotes the annihilator of the center of $\g$.
\end{prop}
\begin{proof} The implication ``$\Leftarrow$'' follows immediately from Thm. \ref{thm:aregpois}. For the other implication, we remark that
for any Lie algebra $\g$ one checks in a straightforward way that $Z(\g)\subset \h_{\xi} $ for all $\xi\in \g^*$.
By Propositions \ref{prop:projE} and \ref{prop:linPois}, $\g^*$ is almost regular if{f} $dim(\h_{\xi})=dim(Z(\g))$ for all $\xi\in \g^*$, i.e. if $\h_{\xi}=Z(\g)$. We conclude using Thm. \ref{thm:aregpois}, recalling that $D=\h^{\circ}$.  
\end{proof}

In particular, when $Z(\g)=\{0\}$, $\g^*$ is almost regular if{f} the union of the open symplectic leaves is dense.
For instance, the Poisson manifold $(\RR^2, x\partial_x\wedge \partial_y)$
-- which is the dual of the two-dimensional Lie algebra with bracket $[e_1,e_2]=e_1$ -- 
 is almost regular, and even more, it is log symplectic.

\subsection{Actions}\label{subsec:actions}

It is tempting to use group actions to construct almost regular Poisson structures. Unfortunately the simplest possible attempt, which we now outline,   only delivers regular Poisson structures.

Let $\g$ be a Lie algebra and $r\in \wedge^2\g$ such that\footnote{We then obtain a triangular Lie bialgebra, given by the Lie algebra structure on $\g$ and by the cocycle $\delta:=[r,\cdot]\colon \g \to \wedge^2 \g$  \cite[\S 2.2, \S 2.4]{Ping}.} $[r,r]=0$.
Let $M$ be any manifold and $\sigma\colon \g \to \vX(M)$ a Lie algebra homomorphism, i.e. an infinitesimal (right) action of the Lie algebra $\g$ on $M$.
The bivector field $\pi:=(\wedge^2\sigma)(r)$ is clearly a Poisson bivector field\footnote{The infinitesimal action of the Lie algebra $\g$ on $(M,\pi)$ -- when integrable -- is then integrated by a Poisson action of the simply connected Poisson-Lie group integrating the above Lie bialgebra.} on $M$.  The image of the contraction map $r^{\sharp}\colon \g^*\to \g$ is a Lie subalgebra $\mathfrak{k}\subset \g$. We do not lose any information by restricting ourselves to the  infinitesimal action of $\mathfrak{k}$ on $M$ (which we still denote by $\sigma$) and regarding $r$ as a  (full rank) element of $\wedge^2\mathfrak{k}$.

Assume that the  transformation algebroid $\mathfrak{k}\times M$ is almost injective. Unfortunately, even with this assumption the Poisson manifold $(M,\pi)$ is \emph{not} almost regular in general. Indeed, the projective singular foliation $\cF_{\mathrm{act}}$ induced by the infinitesimal action (see Ex. \ref{ex:fact}) does not necessarily agree with the singular foliation $\cF_{\mathrm{Pois}}$ induced by the Poisson structure $\pi$ as in \S \ref{subsec:pois}. In general, we have only $$\cF_{\mathrm{Pois}}\subset \cF_{\mathrm{act}},$$
which follows immediately from the relation
$$\sigma_p\circ r^{\sharp} \circ \sigma_p^*=\pi_p^{\sharp}\colon T_p^*M\to T_pM$$
for all $p\in M$, where $\sigma_p\colon \mathfrak{k}\to T_pM$.
Notice that if $ker(\sigma_p)+\mathrm{Im}(r^{\sharp} \circ \sigma_p^*)=\mathfrak{k}$ for all points $p$,
 then the leaves of $\cF_{\mathrm{Pois}}$ and $\cF_{\mathrm{act}}$ coincide.
If the infinitesimal action is almost free, i.e. $\sigma_p$ is injective at every point, then we obtain  $\cF_{\mathrm{Pois}}=\cF_{\mathrm{act}}$, a regular foliation associated to a  regular Poisson structure.
 
A concrete example of the above is the following. 

\begin{ex}
As in Ex. \ref{ex:fact}, to which we refer for the notation, consider the action of $\CC^*$ on $\CC^2$ given by $z\cdot(w_1,w_2)=(zw_1,z^kw_2)$ for some $k\in \ZZ$. Denote by $\sigma \colon \mathfrak{k}\to \vX(\CC^2)$ the corresponding infinitesimal action (here
$\mathfrak{k}$ is the Lie algebra of the Lie group $\CC^*$). Then   $\cF_{\mathrm{act}}=\langle E_1+kE_2, X_1+kX_2\rangle$, a singular foliation generated by linear vector fields. Consider the element $r\in \wedge^2 \mathfrak{k}\cong \RR$ for which\footnote{Identifying $\CC^*\cong (\RR_+,\cdot) \times  S^1$ and using standard coordinates $R$ and $\theta$, we have $r=\partial_R\wedge\partial_{\theta}$.}
 $(\wedge^2\sigma)(r)=(E_1+kE_2)\wedge (X_1+kX_2)=:\pi$. Notice that $\pi$ is a Poisson bivector field that vanishes quadratically at the origin, so the singular foliation $\cF_{\mathrm{Pois}}$   induced by $\pi$ is generated by quadratic vector fields. In particular, $\cF_{\mathrm{Pois}}\subsetneq\cF_{\mathrm{act}}$. Notice  however that the leaves of the two foliations coincide, since the vanishing set of $E_1+kE_2$ agrees with the one of $X_1+kX_2$.

 We saw in Ex. \ref{ex:fact} that $\cF_{\mathrm{act}}$ is projective. On the other hand, when $k\neq 0$, $\cF_{\mathrm{Pois}}$ is not projective (so $\pi$ is not almost regular). This follows from  Lemma \ref{lem:nicechar} i) and {Prop.  \ref{prop:projE}}, for the Poisson structure $\pi$ is regular of rank two at points $p$ away from the origin (so $dim(\h_p)=2$ there), while at the  origin we have $\h_0=0$, since any 1-form that annihilates the vector field $E_1+kE_2$ must vanish at the origin. When $k=0$, we have $\pi=E_1 \wedge X_1=(x_1^2+y_1^2)\partial_{x_1}\wedge \partial_{y_1}$, which is clearly an almost regular Poisson structure on $\CC^2$.
 \end{ex}

\section{The holonomy groupoid is a Poisson groupoid}\label{sec:Poisgpd}

{In this section we show that almost regular Poisson manifolds -- which  in general are not integrated by a symplectic groupoid --  come together with   well-behaved  Lie groupoids, that among other things allow to desingularize them. This is not surprising, since these Poisson manifolds are exactly those whose associated singular foliation is projective, and to projective singular foliations one associates an adjoint Lie groupoid as in Prop. \ref{prop:adj}. More precisely, in \S \ref{subsec:Poisgr} to an almost regular Poisson manifold we associate two Poisson groupoids in duality, both of which are adjoint groupoids, and one of which has a regular Poisson structure. In \S \ref{subsec:duality} we show that these two Poisson groupoids are related by a morphism. An interesting question raised by Y. Kosmann-Schwarzbach is whether each of these groupoids acts on the other one by dressing transformations.}

\subsection{Lie bialgebroids}\label{sec:Liebi}

We recall some standard material on Lie bialgebroids and Poisson groupoids, based on the work of Mackenzie-Xu and Weinstein. We refer the reader to the monograph \cite{Ping} for a review of Poisson groupoids from a modern perspective.

A \emph{Lie bialgebroid} \cite[\S 3]{MacXuBiPois}
 is a pair $(A,A^*)$ where
 $A$ is a Lie algebroid with the property that its dual $A^*$ is also
a Lie algebroid, and both structures are compatible in the following sense:
$$d_{A^*}[X,Y]=[d_{A^*}X,Y]+[X,d_{A^*}Y],$$
where $d_{A^*}\colon \Gamma(\wedge A)\to \Gamma(\wedge A)$
is the differential associated to the Lie algebroid $A^*$,  $X,Y\in \Gamma(A)$, and the square bracket denotes the extension to $\Gamma(\wedge A)$ of the Lie algebroid bracket of $A$.

A  \emph{Poisson groupoid}  \cite[\S 4.2]{alancoiso} is a Lie groupoid $G\rightrightarrows M$ 
endowed with a Poisson structure $\Pi$ such that the graph of the multiplication map is a coisotropic submanifold of $(G,\Pi)\times (G,\Pi)\times (G,-\Pi)$. In that case the identity section $M$ is coisotropic in $(G,\Pi)$. Therefore its conormal bundle $N^*M:=\{\xi \in (T^*G)|_M: \xi|_{TM}=0\}$ is a Lie subalgebroid of $T^*G$, and the canonical isomorphism of Lie algebroids $N^*M\cong [\ker(d\bs)|_M]^*$ 
makes the latter into a Lie algebroid. One can show that the pair $(\ker(d\bs)|_M,[\ker(d\bs)|_M]^*)$ is a Lie bialgebroid, see \cite[Thm. 8.3]{MacXuBiPois}.

Conversely, let $(A,A^*)$ be a  Lie bialgebroid.
If $A$ is integrable, it is known that its $\bs$-simply connected Lie groupoid admits a unique Poisson structure making it a Poisson groupoid that integrates the Lie bialgebroid $(A,A^*)$ \cite[Thm 4.1]{maxu}. In general this Poisson structure does not descend to the other Lie groupoids integrating $A$, which are all quotients of the $\bs$-simply connected one.

A special kind of Lie bialgebroid was introduced in  \cite[\S 4]{MacXuBiPois}.
Let $A$ be a Lie algebroid, and $\Lambda$ a section of $\wedge^2A$ such that $[\Lambda,\Lambda]=0$. The section 
$\Lambda$ induces on the dual vector bundle $A^*$ the structure of a Lie algebroid, and the pair $(A,A^*)$ is a Lie bialgebroid. Further, the map $A^*\to A$ obtained contracting with $\Lambda$ is a Lie algebroid morphism. A typical example of this is obtained setting $A=TM$ for $M$ a Poisson manifold. Lie algebroids that arise as above are called \emph{triangular Lie bialgebroids} and $\Lambda$ is called $r$-matrix.
In the above setting, for any (not necessarily $\bs$-simply connected) Lie groupoid integrating $A$, the formula $\overset{\leftarrow}{\Lambda}-\overset{\rightarrow}{\Lambda}$ defines a Poisson groupoid structure integrating the triangular Lie bialgebroid $(A,A^*)$ \cite[Thm. 3.1]{LiuXuExactPoisGr}. (Here $\overset{\leftarrow}{\Lambda}$ denotes the left-invariant bivector field obtained from $\Lambda$, and 
$\overset{\rightarrow}{\Lambda}$ the right-invariant one.)
The pair $(A^*,A)$ is also a Lie bialgebroid, but in general it is not triangular.

\subsection{The Lie bialgebroid associated with an almost regular Poisson structure}

Let $(M,\pi)$ be an almost regular Poisson manifold.
Clearly the involutive distribution $D$ introduced in \S \ref{sec:almregmfd} is a Lie algebroid over $M$.
Since each leaf $P$ of $D$ is endowed with the Poisson bivector field $\pi_P$ obtained by restriction  which satisfies $[\pi_P,\pi_P]=0$,  $D^*$ is also a Lie algebroid  and $$(D,D^*)$$ a triangular Lie bialgebroid (see \S \ref{sec:Liebi}). This simple observation is the starting point of this whole section.

 We remind that by Remark \ref{rem:Dstar} and Notation \ref{not:algdcF} 
 there is a canonical isomorphism of Lie algebroids  
\begin{equation}\label{eq:dfx}
D^*\cong A^{\cF} 
\end{equation} 
induced by the contraction with $\pi$, 
where $A^{\cF}$   was introduced in   Lemma \ref{lem:algcF} as an   almost injective Lie algebroid inducing $\cF$.

\begin{ex}
Consider $M=\RR^2$ with the almost regular Poisson structure $\pi=q\partial_{q}\wedge \partial_{p}$. We have $D=TM$. 
At   $x=(q,p)\in \RR^2$, the isomorphism \eqref{eq:dfx} is $T_x^*M\to A^{\cF}_{x}$  
 given by $dq\mapsto [q\partial_{p}]$ and $d{p}\mapsto [-q\partial_{q}]$. The $r$-matrix $\pi \in \wedge^2\Gamma(TM)$ can be viewed as a bilinear form on $T_x^*M$, which maps $(dq,dp)$ to $q$. Viewed as a bilinear form on 
$A^{\cF}_{x}$ using the above isomorphism, it maps $([X],[Y])$ to $\omega(X,Y)|_x$ where  $X,Y\in \cF$ and $\omega$ is the symplectic form on the symplectic leaf through $x$. (So $\omega=0$ if 
$q=0$, and $\omega=\frac{1}{q}dq\wedge dp$ otherwise.)
\end{ex}

\subsection{\texorpdfstring{$H(D)$}{H(D)} and \texorpdfstring{$H(\cF)$}{H(F)} as Poisson groupoids}\label{subsec:Poisgr}

Let  $(M,\pi)$ be almost regular Poisson manifold. The material recalled in \S \ref{sec:Liebi} immediately implies:
\begin{prop}\label{prop:HD}
The holonomy groupoid $H(D)$ of the regular foliation associated to $D$, together with the bivector field 
$\overset{\leftarrow}{\pi}-\overset{\rightarrow}{\pi}$,  
is a Poisson groupoid integrating the Lie bialgebroid $(D,D^*)$.
\end{prop}

The above is a statement about  the  Lie bialgebroid $(D,D^*)$.
In the rest of this subsection we will consider {its ``dual'', namely} the Lie bialgebroid $(D^*, D)$.
Since the singular foliation $\cF$ is   projective, 
its holonomy groupoid $H(\cF)$ is a Lie groupoid that integrates the  almost injective Lie algebroid $D^*$, by eq. \eqref{eq:dfx} and \S \ref{subsec:holgr}.

\begin{thm}\label{thm:Poisgr}
Let $(M,\pi)$ be an almost regular Poisson manifold.
Let  $\cF$ be the corresponding singular foliation.
Then there is a canonical Poisson structure $\Pi$ on
the holonomy groupoid $H(\cF)$ making it a 
Poisson groupoid for the Lie bialgebroid $(D^*,D)$.
\end{thm}

\begin{remark}
 Since $H(\cF)$ is not $\bs$-simply connected in general,  we can not make use  of   \cite[Thm 4.1]{maxu} to obtain the above result.  On the contrary, $H(\cF)$ is the adjoint groupoid integrating the almost injective Lie algebroid $D^*$ associated to $\cF$, as stated in \S \ref{subsec:holgr}. It follows that \emph{any} Lie groupoid integrating the Lie algebroid $D^*$ is a Poisson groupoid integrating $(D^*,D)$.
\end{remark}

\begin{proof}
Let $\Gamma \rightrightarrows M$ be the $\bs$-simply connected Lie groupoid 
integrating the almost injective Lie algebroid $D^*$ (it is obtained taking the 
universal covers of the source-fibers of $H(\cF)$). The isomorphism of Lie algebroids \eqref{eq:dfx} and
Prop \ref{lem:explicit} 
 imply that $H(\cF)=\Gamma/\sim_{\Gamma}$. The equivalence relation $\sim_{\Gamma}$  introduced in Prop. \ref{lem:explicit} can be alternatively described as follows\footnote{This follows from \cite[Cor. 2.11 b)]{AndrSk} and \cite[Ex. 3.4(4)]{AndrSk}.}: 
$$y_1\sim_{\Gamma}y_2 
\Leftrightarrow \text{$\exists$ neighbourhood $U$ of $y_1$  and  $\tau \colon U\to \Gamma$ satisfying $\phi\circ \tau=\phi$ and $\tau(y_1)=y_2$}
$$
where $\phi:=(\bt,\bs)\colon \Gamma\to M\times M$ is the   target-source map. We denote by $p\colon \Gamma\to \Gamma/\sim_{\Gamma}=H(\cF)$ the canonical projection.

Since $\Gamma$ is $\bs$-simply connected and integrates the Lie algebroid $D^*$, it admits a unique Poisson structure $\Pi$ making it a Poisson groupoid integrating the Lie bialgebroid $(D^*,D)$ by \cite[Thm 4.1]{maxu}.

{\bf Claim:} \emph{The Poisson structure $\Pi$ on $\Gamma$ descends to a well-defined Poisson structure on    $\Gamma/\sim_{\Gamma}$.} 

Let $y_1,y_2\in \Gamma$ such that $y_1\sim_{\Gamma} y_2$. {We need to show that 
$p_*(\Pi_{y_1})=p_*(\Pi_{y_2})$.}
Let the neighbourhood $U$ of $y_1$  and the smooth map $\tau \colon U\to \Gamma$ be as in the above description of the equivalence relation $\sim_{\Gamma}$. 
Notice that $p\circ \tau =p$, as a consequence of $\phi\circ \tau=\phi$ and the description of $\sim_{\Gamma}$. Hence it suffices to show that $\tau_*(\Pi_{y_1})=\Pi_{y_2}$.

 The situation is summarized by the following commutative diagram:
\begin{equation*}
\xymatrix{
&\Gamma/\sim_{\Gamma}&\\
U \ar[rd]_{\phi} \ar[ru]^{p} \ar[rr]^{\tau} & & \tau(U)\ar[ld]^{\phi}\ar[lu]_{p} \\
&M\times M&\\
}
 \end{equation*} 

We will use the fact that, as for every Poisson groupoid,
$\phi=(\bt,\bs)\colon \Gamma\to M\times M$
is a Poisson map
\cite[Thm. 4.2.3]{alancoiso}, where the Poisson structure on the codomain is the product  $(M,\pi)\times(M,-\pi)$.

Denote by $M_{\mathrm{reg}}$ the open dense subset of $M$ consisting of regular points of 
{the singular foliation $\cF$ induced by $\pi$.} 
$N:=\bs^{-1}(M_{\mathrm{reg}})=\bt^{-1}(M_{\mathrm{reg}})$ is an open dense subset of $\Gamma$,
as the preimage  of $M_{\mathrm{reg}}$ under a surjective submersion.  
For all $z\in U\cap N$, notice that $\tau(z)\in N$, since $\phi\circ \tau=\phi$.  The derivative
$$(\phi_*)_{\tau(z)}\colon T_{\tau(z)}\Gamma\to T_{\phi(z)}(M\times M)$$
is injective. Indeed its kernel $\ker(\bt_*)_{\tau(z)}\cap \ker(\bs_*)_{\tau(z)}$ is trivial since $\Gamma|_{M_{\mathrm{reg}}}$ has discrete isotropy groups, which in turn is due to two facts: {i)} $H(\cF)$ is a discrete quotient of $\Gamma$ and {ii)}
$H(\cF)|_{M_{\mathrm{reg}}}=H(\cF|_{M_{\mathrm{reg}}})$,  being the holonomy groupoid of a regular foliation, has discrete isotropy groups. 
This injectivity implies that $\tau_*(\Pi_{z})=\Pi_{\tau(z)}$, since both sides map to the same bivector under $(\phi_*)_{\tau(z)}$ (recall that $\phi$ is a Poisson map and  $\phi\circ \tau=\phi$).
Since $U\cap N$ is dense in $U$, by continuity the same holds for all points $z\in U$, in particular we have $\tau_*(\Pi_{y_1})=\Pi_{y_2}$.
This proves the claim. \hfill$\bigtriangleup$

Since the fibers of $p$ are discrete, the fact that $(\Gamma,\Pi)$ is a Poisson groupoid integrating $(D^*,D)$ implies that the same is true for $ \Gamma/\sim_{\Gamma}$ with the induced Poisson structure. 
\end{proof}

 For every leaf $P$ of $D$, the Lie algebroid of
 $H(\cF)|_P$ is $D^*\cong T^*P$. The following results is no surprise, and justifies why the Poisson groupoid $H(\cF)$ can be regarded as a ``desingularization'' of the almost regular Poisson manifold $(M,\pi)$:

\begin{cor}\label{cor:desing} For every leaf $P$ of $D$, the restriction 
$H(\cF)|_P$ is a symplectic leaf of the Poisson structure $\Pi$ on $H(\cF)$, which hence is a regular Poisson structure. Further  $H(\cF)|_P$
is a symplectic groupoid for $(P,\pi_P)$.
\end{cor}
\begin{remark}
On the other hand, the restriction of $H(D)$ to each leaf $P$ of $D$ is not a symplectic groupoid in general, as is evident in the case of Poisson structures which have full rank on an open dense subset.
\end{remark}
\begin{proof}
Let    $\Gamma \rightrightarrows M$ be the $\bs$-simply connected Lie groupoid obtained taking the universal covers of the source-fibers of $H(\cF)$. In the proof of Thm. \ref{thm:Poisgr} we saw that the Poisson structure on $H(\cF)$ is obtained by quotienting a Poisson structure $\Pi$ making $(\Gamma,\Pi)$ a Poisson groupoid integrating the Lie bialgebroid $(D^*,D)$, hence we may work on $\Gamma$ instead than on $H(\cF)$.

The unique Poisson structure on $M$ such that the target map  $\bt\colon \Gamma \to M$ is a Poisson map \cite[Thm. 4.2.3]{alancoiso} is  $\pi$,  since up to a sign it agrees with the one induced by the Poisson groupoid structure on $H(D)$   integrating the Lie algebroid $(D,D^*)$ \cite[Thm. 4.4.4]{alancoiso}. 

{\bf Claim:} {\it For each leaf $P$ of $D$,  the restricted groupoid $\Gamma|_P:=\bs^{-1}(P)=\bt^{-1}(P)$ is a Poisson submanifold of $(\Gamma, \Pi)$.}

{For any $y\in \Gamma|_P$ we have linear isomorphisms $[T_{\bt(y)}P]^{\circ}\cong [T_y(\Gamma|_P)]^{\circ}\cong [T_{\bs(y)}P]^{\circ}$ induced by  the submersions  $\bt$ and $\bs$, where the annihilators are taken in $T_{\bt(y)}M$, $T_y\Gamma$ and $T_{\bs(y)}M$ respectively. We have 
\begin{equation*}
d_y\bt\circ \Pi_y^{\sharp}\circ(d_y\bt)^*\xi=\pi_{\bt(y)}^{\sharp}\xi=0\;\;\;\;\; \text{ for all } \xi \in [T_{\bt(y)}P]^{\circ},
\end{equation*} since $\bt$ is a Poisson map and $P$ a Poisson submanifold. The analog equation holds for $\bs$. As a consequence, the restriction to ${[T_y(\Gamma|_P)]^{\circ}}$ of 
$d_y\phi\circ \Pi_y^{\sharp}$ vanishes identically, where $\phi:=(\bt,\bs)\colon \Gamma\to M\times M$.  In the proof of Thm. \ref{thm:Poisgr} we saw that there is an open dense subset $N\subset \Gamma$ where $d\phi$ is injective.
In the case that $y$ lies in $N$,
it follows that $\Pi_y^{\sharp}$ annihilates ${[T_y(\Gamma|_P)]^{\circ}}$. Since $N$ is a dense subset of $\Gamma$, this holds at all $y\in \Gamma|_P$, proving the claim. \hfill$\bigtriangleup$
}

$\Gamma|_P$, with the restriction $\Pi|_{\Gamma|_P}$, is a Poisson groupoid\footnote{Indeed, the graph of the multiplication of $\Gamma|_P$
is the intersection of the graph of the multiplication of $\Gamma$
with the Poisson submanifold $\Gamma|_P\times \Gamma|_P\times \Gamma|_P$, and therefore it is coisotropic in the latter.}
integrating the Lie bialgebroid $(D|_P,D^*|_P)=(T^*P,TP)$.
 By the uniqueness statement in \cite[Thm 4.1]{maxu}, $\Pi|_{\Gamma|_P}$ agrees with the inverse of the symplectic form making $\Gamma|_P$ the $\bs$-simply connected  symplectic groupoid integrating $T^*P$.
\end{proof}
 
\begin{remark}
Given any Poisson manifold $(M,\pi)$ (not necessarily almost regular), choose a regular foliation of $M$ consisting of Poisson submanifolds. If we
 denote by $D$ the distribution tangent to the chosen regular foliation, then $(D,D^*)$ is a Lie bialgebroid. In general $D^*$ is not an integrable Lie algebroid. When it is,  Lie groupoids integrating it and which are not $\bs$-simply connected generally do not admit Poisson groupoid structures integrating the Lie bialgebroid $(D^*,D)$. This explains the relevance of Thm. \ref{thm:Poisgr}, which holds in the almost regular case.

We present a well-known example of the above (see for instance \cite[Ex. 5.16]{LecturesIntegrabilty}). 
Endow any manifold $M$ with the trivial Poisson structure $\pi=0$. As a regular foliation on $M$ take the one-leaf foliation, so $D^*=T^*M$ with the trivial Lie algebroid structure. The $\bs$-simply connected symplectic groupoid integrating $(M,\pi)$ is $T^*M$ endowed with the canonical symplectic form $\Omega$ and with groupoid composition being addition along the fibers. Now specialize to  $M=S^3$, the 3-sphere. Then $T^*S^3\cong S^3\times \RR^3$, and quotienting by the bundle of lattices $S^3\times \ZZ^3$ we obtain another Lie groupoid integrating the cotangent bundle Lie algebroid, namely $G=S^3\times \T^3$ where $\T^3$ is the 3-torus. The symplectic form $\Omega$ however does \emph{not} descend to $G$. Indeed, $G$ does not admit any symplectic structure at all. To see this, suppose $G$   admitted a symplectic form $\omega$, defining a cohomology class $[\omega]\in H^2(G)\cong
H^0(S^3)\otimes H^2(\T^3)$ (we used the K\"unneth formula here). We have $[\omega]^2
\in H^0(S^3)\otimes H^4(\T^3)=\{0\}$, therefore the top power of the symplectic form satisfies $[\omega]^3=0$, which is a contradiction since $G$ is compact.
\end{remark}

We spell out the special case of Thm. \ref{thm:Poisgr} in which the regular foliation on $M$ consists of just one leaf,
 generalizing \cite[Cor. 2.19]{GuLi}. 
 \begin{cor}
Let $M$ be a manifold and $\pi$ be a Poisson structure  which is symplectic on an open dense subset of $M$ (for instance, a log symplectic structure). Consider the Lie algebroid $T^*M$ associated to the Poisson structure. Then the adjoint groupoid of $T^*M$ 
is a symplectic groupoid for the Poisson manifold $(M,\pi)$. 
\end{cor}
   
\subsection{Duality of Poisson groupoids}\label{subsec:duality}

It is well-known  \cite[\S 3]{MacXuBiPois} that if $(A,A^*)$ is a Lie bialgebroid, then $(A^*,A)$ also is.
A Poisson groupoid integrating a Lie bialgebroid $(A,A^*)$ and one integrating the 
Lie bialgebroid $(A^*,A)$ are said to be \emph{dual} to one another \cite[Def. 4.4.1]{alancoiso}.

Given an almost regular Poisson structure $(M,\pi)$, 
we saw that
 $H(D)$ and $H(\cF)$ are dual Poisson groupoids, by Prop. \ref{prop:HD} and Thm. \ref{thm:Poisgr}. The Poisson structure they induce on $M$ (via the target map) are equal up to a sign and are given by $-\pi$ and $\pi$ respectively.  Notice that the orbits on $M$ induced by these groupoids differ, since they are given   respectively by  the leaves of $D$ and the leaves of $\cF$ (i.e. the symplectic foliation of $M$).

Since the Lie bialgebroid $(D,D^*)$ arises from the $r$-matrix $\pi$, by \S \ref{sec:Liebi} we have a Lie algebroid morphism
$\sharp \colon D^*\to D$.

\begin{prop}\label{prop:hfhd}
There is a morphism of Lie groupoids
\begin{equation}\label{eq:fp}
H(\cF)\to H(D)
\end{equation}
which integrates $\sharp$ and is an anti-Poisson map.
\end{prop}
\begin{remark}
The general theory of Lie groupoids implies that $\sharp$   integrates to a morphism $\Gamma\to H(D)$ where $\Gamma$ is  the $\bs$-simply connected Lie groupoid integrating  $D^*$. 
This does not imply the above proposition, since 
  $H(\cF)$ is not $\bs$-simply connected in general. 
\end{remark}

\begin{proof}
The map
$\sharp \colon D^*\to D$
is   a morphism of Lie \emph{bialgebroids} -- i.e. both a Lie algebroid morphism and a Poisson map -- from the Lie bialgebroid $(D^*,D)$ to the Lie bialgebroid $(D, \bar{D^*})$, where $\bar{D^*}$ denotes the Lie algebroid structure on the dual of $D$ given 
induced by $-\pi$ (see \S \ref{sec:Liebi}). Indeed, it is a Lie algebroid morphism by \S \ref{sec:Liebi}, and a Poisson map because its dual $-\sharp$ is a Lie algebroid morphism from $\bar{D^*}$ to $D$.
 Therefore $\sharp$ integrates to a morphism of Poisson groupoids $$\Gamma\to\Delta,$$ where $\Gamma$ is the $\bs$-simply connected Poisson groupoid integrating $(D^*,D)$  and $\Delta$  the $\bs$-simply connected Poisson groupoid integrating the triangular Lie bialgebroid $(D, \bar{D^*})$. Equivalently,  $\Delta$ is the Poisson groupoid obtained after changing the sign of the Poisson structure on the $\bs$-simply connected Poisson groupoid integrating the triangular Lie bialgebroid $(D,D^*)$.

Since $\cF$ and $D$ are projective foliations, we can use the results of Debord  \cite{DebordJDG}.  In \cite[Thm. 1, p. 492]{DebordJDG} it is shown that the holonomy groupoid associated to a projective foliation is a quasi-graphoid 
(for the definition see {\S \ref{subsec:holgr}), so in particular $H(D)$ is a quasi-graphoid}. 
Now, since $H(\cF)$ 
integrates the Lie algebroid $D^*$,
there is a  neighbourhood $V$ of the identity section in $H(\cF)$ which is isomorphic to a neighbourhood of the identity section in $\Gamma$. 
The composition $\phi\colon V\hookrightarrow \Gamma\to\Delta\to H(D)$  commutes with the source and target maps.
Since $H(\cF)$ is $\bs$-connected and $H(D)$ a quasi-graphoid, we can 
apply \cite[Prop. 1, p. 474]{DebordJDG} to conclude that there is a unique extension $\Phi : H(\cF) \to H(D)$ of $\phi$ which is a morphism of Lie groupoids. 

$\Phi$ differentiates to $\sharp$, and is an anti-Poisson map since $V\hookrightarrow\Gamma\to \Delta$ is a Poisson map and $\Delta\to H(D)$ is an anti-Poisson map.
\end{proof}

\subsection{Examples}

\begin{ex}

\begin{enumerate}
\item[i)]   Assume that $\pi$ is a regular Poisson structure on $M$.
 Then $\sharp \colon D^*\to D$ is an isomorphism of Lie algebroids \cite[Def. 4.4.2 (f)]{alancoiso}, and it integrates to an isomorphism of Lie groupoids $$H(\cF)\cong H(D).$$

\item[ii)] Assume that $\pi$ has full rank on a dense subset of $M$, i.e. $D=TM$.  Then
$H(D)=M\times M$,  the pair groupoid,  with the Poisson structure $(-\pi)\times \pi$ (see Prop. \ref{prop:HD}), and  $H(\cF)$ is a symplectic groupoid. The   morphism
induced by   the anchor map $\sharp \colon T^*M\to TM$  
 is  $$(\bt,\bs)\colon H(\cF)\to M\times M$$
  is given by the target and source maps of $H(\cF)$, and is an anti-Poisson map. 
  \end{enumerate}
\end{ex}
 
\begin{remark}\label{sec:discussion}
A particular case of ii) above is 
when $(M,\pi)$ is a log symplectic structure which is proper (see \cite[Def. 1.10]{GuLi}).

In that case Gualtieri-Li obtained\footnote{The Lie groupoid $H(\cF)$ is isomorphic to $\mathrm{Pair}_{\pi}(M)$ as introduced in \cite{GuLi}, because both $H(\cF)$ and $\mathrm{Pair}_{\pi}(M)$ are the adjoint groupoid of the cotangent Lie algebroid $T^*M$. This follows from Prop. \ref{prop:adj} and
  \cite[Cor. 2.19]{GuLi}.} 
 $H(\cF)$ -- as a Poisson groupoid -- via a blow-up procedure applied to the Lie groupoid $\mathrm{Pair}_Z(M)  {\rightrightarrows M}$ \cite[Thm. 2.15]{GuLi}, where $\mathrm{Pair}_Z(M)$ is isomorphic to the holonomy groupoid of the projective singular foliation of vector fields tangent to the exceptional hypersurface $Z$ of $(M,\pi)$ \cite[Cor. 2.18]{GuLi}. Further
 $\mathrm{Pair}_Z(M)$ itself is obtained from a blow-up procedure  from {the pair groupoid} $M\times M$ \cite[Thm. 2.14]{GuLi}.
Finally,    the morphism 
$(\bt,\bs)$ above
factors through the blow-down maps and morphisms of Lie groupoids $$H(\cF)\to \mathrm{Pair}_Z(M)\to M\times M.$$
 
This prompts us to wonder whether there is a larger class of  almost regular Poisson structures for which the Poisson groupoid $H(\cF)$ is obtained from $H(D)$ by some kind of iterated blow-up procedure. Natural candidates are certain classes of  log-f symplectic structures as introduced in \S \ref{sec:fls}, {see also Remark \ref{rem:explosion} below.}
\end{remark}

\paragraph{Linear Heisenberg-Poisson structures}\

For the next class of examples we need some preparation. We omit the proof of the following lemma.
\begin{lemma}\label{lem:sympgr}
Let $(V,\pi)$ be a Poisson vector space (i.e. $\pi\in \wedge^2 V$).

1) A symplectic groupoid integrating the Poisson  manifold $(V,\pi)$ is
$V^*\ltimes V$, with
\begin{itemize}
\item the action groupoid structure induced by the action of the Lie group $(V^*,+)$ on $V$ as follows: $\xi \in V^*$ acts translating by $\pi^{\sharp}\xi$,
\item the constant symplectic form $\Omega_{\pi}$ defined\footnote{Equivalently, using the canonical identification $V^*\times V\cong T^*V^*$, it is given by minus the canonical symplectic form on $T^*V^*$ plus a ``magnetic term'', namely the pullback of $\pi$ viewed as a 2-form on $V^*$.}
 by  $((\xi,v), (\xi',v'))\mapsto 
\langle v, \xi'\rangle-\langle v', \xi\rangle+\pi(\xi,\xi')$.
\end{itemize}
2) When $\pi$ has full rank, denoting by  $\omega\in \wedge^2V^*$ the corresponding\footnote{That is: the map $V\to V^*, v\mapsto \iota_v\omega$ equals $-\pi^\sharp$.} symplectic form, the following is an isomorphism of symplectic groupoids:
$$V^*\ltimes V \to V\times V,\;\;\;\ (\xi,v)\mapsto (v+\pi^{\sharp}\xi, v),$$
where $V^*\ltimes V$ is as in 1) and $V\times V$ is the pair groupoid endowed with symplectic form $\omega\times(-\omega)$.
\end{lemma}

Using the above lemma we can describe explicitly the Poisson groupoid $H(\cF)$
for a class of almost regular Poisson manifolds.

\begin{prop}\label{prop:exft} 
Let $V$ be a vector space and $\pi\in \wedge^2 V$ a bivector with full rank, i.e. one obtained from a symplectic form. Let $f\colon \RR\to \RR$ a smooth function such that $\{t:f(t)\neq 0\}$ is dense in $\RR$. Consider the almost regular Poisson manifold $$\left(V\times \RR,\;\Pi:=f(t)\pi\right)$$ and denote by $\cF$ its symplectic foliation.

Then \begin{equation}\label{eq:isop}
H(\cF)\cong V^*\ltimes V\times \RR
\end{equation}
 as a Poisson groupoid, where for $V^*\ltimes V\times \RR$
\begin{itemize}
\item the action groupoid structure is induced by the action of the Lie group $(V^*,+)$ on $V\times \RR$ as follows: $\xi \in V^*$ acts   translating by $(f(t)\pi^{\sharp}\xi,0)$,
\item the symplectic leaves are  $(V^*\ltimes V\times \{t\}, \Omega_{f(t)\pi})$ as introduced in Lemma \ref{lem:sympgr}, where $t$ ranges over all real numbers.
\end{itemize}
\end{prop}

\begin{proof}
Since $f$ is a smooth function, the groupoid structure and Poisson structure on $V^*\ltimes V\times \RR$ are smooth. They are compatible and combine into a Poisson groupoid by Lemma \ref{lem:sympgr} 1). 

Notice that, with the notation of Thm. \ref{thm:aregpois}, we have $D_{(v,t)}=T_vV\times \{0\}$ at every point $(v,t)\in V\times \RR$.
The fact that $V^*\ltimes V\times \RR$, as a Lie groupoid, integrates the Lie algebroid $D^*$ follows from Lemma \ref{lem:sympgr} 1).  Further
 $V^*\ltimes V\times \RR$ is a quasi-graphoid as defined in \S \ref{subsec:holgr}: this is true by a continuity argument {since $f$ has full support} and since  the Lie subgroupoids $V^*\ltimes V\times \{t\}$ for which $f(t)\neq 0$ {are quasi-graphoids}, by Lemma \ref{lem:sympgr} 2). By \S \ref{subsec:holgr}, this establishes eq. \eqref{eq:isop} as an isomorphism of  Lie groupoids.

Recall that
for any symplectic groupoid, the  conormal bundle of the (Lagrangian) submanifold of units is mapped isomorphically onto the tangent bundle of the space of units. Applying this to $V^*\ltimes V\times \{t\}$
for every $t$, we see that the  conormal bundle of the space of unit
$V\times \RR$ in $V^*\ltimes V\times \RR$ is isomorphic to $D$ {as Lie algebroid}.
We conclude that $V^*\ltimes V\times \RR$ is a Poisson groupoid that integrates the Lie bialgebroid $(D^*,D)$.  By the uniqueness in \cite[Thm 4.1]{maxu} (see \S \ref{sec:Liebi}) this concludes the proof.
\end{proof}

In the setting of Prop. \ref{prop:exft}, we have $H(D)\cong  V\times V\times \RR$
as Poisson groupoids, where the Poisson tensor at a point $(w,v,t)$ is $-f(t)\pi_w+f(t)\pi_v$ and, 
for each $t$, $V\times V\times \{t\}$
is a Lie subgroupoid with the pair groupoid structure. In accordance with Prop. \ref{prop:hfhd}, the Lie algebroid morphism
$\Pi^{\sharp}$ integrates to a morphism of Lie groupoids and anti-Poisson map
\begin{eqnarray*}
V^*\ltimes V\times \RR&\to V\times V\times \RR,\\ (\xi,v,t)&\mapsto (v+f(t)\pi^{\sharp}\xi, v,t) 
\end{eqnarray*}

essentially given by the target-source map. Notice that this map is not an isomorphism if the function $f$ has a zero.

\begin{remark}\label{rem:explosion}
{In the  case $f(t)=t$, the almost regular Poisson manifolds of Prop. \ref{prop:exft} are instances of   
 the Heisenberg-Poisson manifolds arising from symplectic manifolds   considered in  Ex. \ref{ex:simple2}, and also of log-f symplectic manifolds. 
The Lie groupoid $H(\cF)$ described in Prop. \ref{prop:exft} is isomorphic to the tangent groupoid of  Connes associated to the manifold $V$, i.e. to the union over  all $t\in \RR\setminus\{0\}$ of copies of the pair groupoid $V\times V$ and  at $t=0$ of the vector bundle $TV$ (seen as a Lie groupoid over $V$). This
 can be checked directly using Lemma \ref{lem:sympgr} (use item 2) at $t\neq 0$ and
item 1) at $t=0$).} In general, as pointed out to us by E. Hawkins, for the Heisenberg-Poisson manifold of any symplectic manifold $M$, the associated singular foliation has  holonomy groupoid isomorphic to the  tangent groupoid of $M$.

{  
{More generally, for any Poisson manifold $(M,\pi)$, consider the singular foliation (denoted by $\cF'$) of the  Heisenberg-Poisson manifold $(M\times \RR,t\pi)$. The holonomy groupoid $H(\cF')$, a topological groupoid over $M\times \RR$, was given in \cite[\S 3]{IakAnal} generalizing a construction called  ``deformation to the normal cone'' or ``explosion''. If $\pi$ is almost regular then $t\pi$ is also almost regular, and $$H(\cF')=A^{\cF}\times \{0\}\;\;\sqcup\;\; H(\cF)\times (\RR\setminus\{0\})$$
where $\cF$ is the singular foliation on $M$ induced by $\pi$. Thm. \ref{thm:Poisgr}  implies that $H(\cF')$ is a Poisson groupoid.}
Notice that symplectic groupoids for Heisenberg-Poisson manifolds arising from symplectic manifolds were constructed by Weinstein in  \cite{w-heis} using ``double explosions'', and the construction was extended by Hawkins  \cite[\S 2, \S 4.3]{HawkinsPlanck} to more general  Heisenberg-Poisson manifolds. 
}
\end{remark}

\section{On the role of the holonomy covers in the integrability problem}\label{sec:integralgds}

Let $E \to M$ be an arbitrary Lie algebroid and put $\cF$ the singular foliation associated to $E$, namely $\cF$ is the $C^{\infty}(M)$-submodule of $\vX_c(M)$ generated by the image of $\Gamma_c (E)$ by the anchor map.  Recall that the integrability problem for $E$ consists of determining whether there exists an $\bs$-connected Lie groupoid  whose Lie algebroid is isomorphic to $E$. In this section we discuss the role of the \emph{germinal isotropy Lie algebras} $\h_x$ {(see \S \ref{sec:sfLA})} and the \emph{holonomy covers}\footnote{Let $L_x$ be the leaf of $\cF$ at $x \in M$. The holonomy cover of $L_x$ is the $\bs$-fiber $H(\cF)_x$ of the holonomy groupoid} of leaves of $\cF$ in the integrability problem for $E$. 

Let us start with this by exhibiting some evidence which motivates us to look at the Lie algebras $\h_x$ in the first place. Recall from \cite[Def. 3.1]{CrFeLie} that the   obstruction to the integrability of $E$
is described by the monodromy groups: At every  point $x \in M$, the monodromy group at $x$ is a certain subset $N_x(E)$ of the center $Z\g_x$ of the isotropy Lie algebra $\g_x$ 
(the latter was defined through the extension \eqref{eqn:isotropyLalgd}).
By \cite{CrFeLie}, the main\footnote{The other requirement is that the monodromy groups are discrete in a uniform way as $x$ moves in the base manifold $M$.} requirement for the integrability of $E$ is the discreteness of $N_x(E)$ for every $x \in M$ (first integrability obstruction). It turns out that these monodromy groups, at least nearby the origin, are contained in the germinal isotropy Lie algebras:

\begin{prop}\label{prop:questionnew}
Let $E\to M$ be any  Lie algebroid. For every $x\in M$ there is an open neighbourhood $U$ of $x$ in $E$ such that  $N_x(E)\cap U$ lies inside the germinal isotropy Lie algebra $\h_x$.
\end{prop}
\begin{proof}
By \cite[Prop. 5.10]{CrFeLie},  there is an open neighbourhood $U$ of $x$ in $E$ such that  every element $v\in N_x(E)\cap U$  can locally be extended to a smooth section $\sigma$ of $E$ taking values in the monodromy groups. Since the monodromy group at any point $y$ lies inside $ker(\rho_y)=\g_y$, in particular $\sigma$ takes values in the kernel of the anchor map $\rho$, so $v\in \h_x$ by Lemma \ref{lem:nicechar}.
\end{proof}

\begin{remark}\label{rem:essZ}
As we saw in Prop. \ref{lem:center}, when $E = T^{\ast}M$, where $(M,\pi)$ is a Poisson manifold, the germinal isotropy Lie algebra $\h_x$ lies in the center $Z\g_x$. In this case Prop. \ref{lem:center} and Prop. \ref{prop:questionnew} can be summarized by the following inclusions: for every $x\in M$ there is an open neighbourhood $U$ of $x$ in $T^*M$ such that $$N_x({T^*M})\cap U\subset \h_x\subset Z(ker(\rho_x)).$$  
\end{remark}

\subsection{Extensions of Lie algebroids and the first integrability obstruction}
\label{subsec:extensions}

In order to understand better the role of the germinal isotropy Lie algebras and the holonomy covers in the integrability problem of the Lie algebroid $E$, let us change our point of view slightly: Recall from \cite{CrFeLie} that the monodromy group $N_x(E)$ is isomorphic to a subgroup of the image $\mathcal{M}_x$ of a certain boundary map $\partial_x : \pi_2(L_x) \to \g_x$, where $L_x$ is the leaf through $x$ and $\g_x$ the (usual) isotropy Lie algebra at $x$, namely the kernel of the anchor map $\rho_x : {E_x} \to T_x L_x$. Also recall that $N_x(E)$ is discrete if{f} $\mathcal{M}_x$ is discrete. 

\begin{enumerate}
\item In favorable cases\footnote{That is, when the isotropy Lie algebra $\g_x$ is abelian.}, the boundary map is calculated as follows: Let $\sigma : T_x L_x \to E_x$ be a right-splitting of the anchor map $\rho_x$ and put $R_{\sigma} : T_x L_x \times T_x L_x \to \g_x$ its curvature. Then  $\partial[\gamma] = \int_{S^2}{\gamma^{\ast}}R_{\sigma}$ for every $[\gamma] \in \pi_2(L_x)$. This shows that the integrability of $E$ is strongly related with the understanding of each fiber $E_x$ as an extension $0 \to \g_x \to E_x \stackrel{\rho_x}{\longrightarrow} T_x L_x \to 0$.
\item On the other hand, as we saw earlier in the sequel (\cf eq. \eqref{eqn:cleanextn}) 
 the restriction $E_{L_x}$ is also an extension
\begin{eqnarray}\label{eqn:Brahicleaf}
0 \to \h_{L_x} \to E_{L_x} \stackrel{\hat{\rho}}{\longrightarrow} {\cA^{\cF}_{L_x}} \to 0
\end{eqnarray}
{Recall from \cite{Debord2013} that $H(\cF)_{L_x} = \bs^{-1}(L_x)$ is a (transitive) Lie groupoid and its Lie algebroid is $\cA^{\cF}_{L_x}$.}
\item As we saw in \S \ref{subsec:pois} though, when the foliation $\cF$ is projective we can write the entire Lie algebroid $E$ as an extension 
\begin{eqnarray}\label{eqn:Brahicareg}
0 \to \h \to E \stackrel{\hat{\rho}}{\longrightarrow} A^{\cF} \to 0
\end{eqnarray}
As $A^{\cF}$ is the Lie algebroid of the holonomy groupoid $H(\cF)$, this extension automatically carries all the germinal isotropy Lie algebras $\h_x$ as well as the holonomy covers $H(\cF)_x:=\bs^{-1}(x)$ of the leaves $L_x$. It is straightforward that $E$ is integrable if and only if this extension is integrable.
\end{enumerate}

Whence the integrability problem  for Lie algebroids can be cast to the integrability problem for extensions of Lie algebroids. The latter was treated by Brahic in \cite{Brahic}, so we can exploit the apparatus given there. Let us discuss the calculation of the (first) integrability obstruction for the extension \eqref{eqn:Brahicleaf}.

\begin{enumerate}
\item As we said above, the restriction $H(\cF)_{L_x}$ is a Lie groupoid and it integrates the Lie algebroid $A^{\cF}_{L_x}$. {Put $H(\cF)_x=\bs^{-1}(x)$.} Since \eqref{eqn:Brahicleaf} is a ``clean'' extension (see Remark \ref{rem:cleanextn}) in the sense of Brahic \cite{Brahic}, the obstruction to its integrability is described by {the behaviour, as $x$ moves along $M$, of the groups} $\mathcal{\widehat{M}}_x$ {defined as the image} of a boundary map $\widehat{\partial}_x : \pi_2(H(\cF)_x) \to \h_x$.
\item Now assume that the kernel $\h_{L_x}$ of the extension \eqref{eqn:Brahicleaf} is abelian. The arguments in \cite[Lemma 3.6]{CrFeLie} can be adapted to the setting of \cite[\S 4]{Brahic}. Thus extension \eqref{eqn:Brahicleaf} can be used to calculate the groups $\mathcal{\widehat{M}}_x$ in the following way: Choose a splitting $\omega : A^{\cF}_{L_x} \to E_{L_x}$ and consider its curvature $R_{\omega} : A^{\cF}_{L_x} \times A^{\cF}_{L_x} \to \h_{L_x}$ defined by $R_{\omega}(X,Y) = \omega[X,Y]-[\omega(X),\omega(Y)]$. Then 
\begin{eqnarray}\label{eq:moncomp}
\mathcal{\widehat{M}}_x = \left\{\int_{\gamma}R_{\omega} : [\gamma] \in \pi_2(H(\cF)_x)\right\}
\end{eqnarray}
\item As mentioned in Remark \ref{rem:essZ}, the assumption that $\h_{L_x}$ be abelian is satisfied when $E=T^{\ast}M$ for $(M,\pi)$ an arbitrary Poisson manifold. Consider for instance the Lie-Poisson structure of $\mathfrak{su(2)}^*$.  In Example \ref{ex:su2} we showed that the germinal isotropy Lie algebra vanishes at $0$, whence the monodromy group at $0$ vanishes. The germinal isotropy $\h_x$ at every other point $x \neq 0$ is equal to the (usual) isotropy $\g_x$ of the Lie algebroid $E=T^{\ast}\mathfrak{su(2)}^*$ because the leaf at $x$ is regular (\cf extension \eqref{eqn:essisotropyextn}). Since regular leaves leaf are spheres it follows that $\g_x = \R$.
\end{enumerate}

Last, let us discuss the relation between the monodromy groups $\mathcal{M}_x$ and $\mathcal{\widehat{M}}_x$. Since $H(\cF)_{L_x}$ is a Lie groupoid, we have an $H(\cF)_x^x$-principal bundle $\bt\colon H(\cF)_x \to L_x$ (\cf \cite[Prop. 4.17]{AZ1}) , so we obtain a map $\bt_{\ast} : \pi_2(H(\cF)_x) \to \pi_2(L_x)$. This map is injective, as one sees looking at the associated long exact sequence of homotopy groups.
 The interested reader may check the definitions of the boundary maps $\partial_x$ and $\widehat{\partial}_x$ given in \cite{CrFeLie} and \cite{Brahic} respectively, to verify that $\widehat{\partial}_x = \partial_x \circ \bt_{\ast}$. Hence {$\mathcal{\widehat{M}}_x$ is always a subgroup of $\mathcal{M}_x$.} 

By the injectivity of $\bt_{\ast}$ above, we see that  determining $\mathcal{\widehat{M}}_x$ requires in principle less computations than determining $\mathcal{M}_x$.

\subsection{Conclusions}\label{sec:conclusions}

The discussion in the previous section \S \ref{subsec:extensions} concludes that thanks to:
\begin{itemize}
\item the notion of \textit{germinal isotropy};
\item working with \textit{holonomy covers of leaves} instead of the leaves themselves,
\end{itemize}
we can establish the following facts:
\begin{enumerate}
\item\label{concl:a} \textit{If $E \to M$ is a Lie algebroid with \emph{abelian} germinal isotropy at each point $x \in M$, the {monodromy groups $\widehat{\mathcal{M}}_x$ can be computed as in \eqref{eq:moncomp}.}}
\item\label{concl:b} \textit{The {monodromy groups $\widehat{\mathcal{M}}_x$} of any Poisson manifold $(M,\pi)$ can be computed as in \eqref{eq:moncomp}.}
\item\label{concl:c} \textit{{Let $(M,\pi)$ be an arbitrary Poisson manifold. If the holonomy covers of the leaves of the associated symplectic foliation are 2-connected, i.e. $\pi_{2}(H(\cF)_x) =\{0\}$ for all $x\in M$, then the Lie algebroid $T^* M$ is integrable.}}
\end{enumerate}
{Item (\ref{concl:c}) above follows because when all the holonomy covers are 2-connected, the monodromy groups $\widehat{\mathcal{M}}_x$ vanish for every $x \in M$. For instance, this is exactly what happens for the example given in Prop. \ref{prop:exft}.}

\paragraph{The almost regular case} 

\begin{remark}
{Note that when $(M,\pi)$ is an almost regular Poisson structure, item (\ref{concl:c}) above also follows from \cite[Thm. 5]{VanEst}. Indeed, in this case we have an abelian extension of Lie algebroids
\begin{eqnarray}\label{eqn:extn}
0\to \h\to T^*M\overset{\widetilde{\rho}}{\to} D^{\ast} \to 0
\end{eqnarray}
where ${D^{\ast}}\cong A^{\cF}$ (\cf Remark \ref{rem:Dstar}, Notation \ref{not:algdcF} and Lemma \ref{lem:algcF} for the Lie algebroid structure of $D^{\ast}$). Since $\h$ is abelian, \cite[Thm. 5]{VanEst} implies that if $D^{\ast}$ integrates to a Lie groupoid with 2-connected $\bs$-fibers then $T^*M$ is integrable. Now use that a Lie groupoid integrating $D^{\ast}$ is the holonomy groupoid $H(\cF)$, by \S \ref{subsec:holgr}, in particular Prop. \ref{prop:adj}.}

{Whence item (\ref{concl:c}) above can be thought of as a generalization of \cite[Thm. 5]{VanEst} to arbitrary Poisson manifolds.}
\end{remark}

{Now suppose that $(M,\pi)$ is an almost regular Poisson manifold such that $T^{\ast}M$ is integrable. Let $\Sigma$ the source simply connected (s.s.c.) Lie groupoid integrating $T^{\ast}M$. It turns out that the topology of the $\bs$-fibers of $\Sigma$ is related with the topology of the holonomy covers. Specifically:}

{Consider the short exact sequence of Lie groupoids
\begin{eqnarray}\label{eq:sesgroid} 
1_M\to \ker p \to \Sigma \stackrel{p}{\to} H(\cF)\to 1_M
\end{eqnarray}
integrating \eqref{eqn:extn}. Notice that $\ker p$ is a bundle of abelian Lie groups, possibly disconnected and of dimension equal to $rank(\h)$. For every $x \in M$, we denote $(\ker p)_x^0$ the connected component of the identity of $(\ker p)_x$.}

\begin{prop}\label{prop:almregintegr}
{Let $(M,\pi)$ be an almost regular Poisson manifold. If $T^{\ast}M$ is integrable, then at every $x\in M$ 
there is a short exact sequence of groups
\begin{equation}\label{eq:sesgr}
1 \to \pi_2(\Sigma_x) \to \pi_{2}(H(\cF)_x) \to 
\pi_1((ker p)_{x}^{0})
 \to 1.
\end{equation}
In particular, if the source fibers of $H(\cF)$ are 2-connected, then 
$T^*M$ is an integrable Lie algebroid and 
the source fibers of its s.s.c. integration $\Sigma$ are also 2-connected.}
\end{prop}

\begin{remark}
{It is a classical fact about covering spaces that $\pi_1((\ker p)_{x}^{0})$ is isomorphic to 
the kernel of the universal covering map (and group homomorphism) $\RR^{\mathrm{dim}(\h)} \to (\ker p)_{x}^{0}$.}
\end{remark}

\begin{proof}
{Fix $x\in M$. The $\bs$-fibers at $x$ of the short exact sequence \eqref{eq:sesgroid} constitute the fibration $1 \to (\ker p)_x \to \Sigma_x \stackrel{p}{\to} H(\cF)_x \to 1$. Apply  the long exact sequence of homotopy groups to this fibration, and use that $\pi_2((\ker p)_x)=\{0\}$ (since $\pi_2$ of any Lie group is trivial) and that $\pi_1(\Sigma_x)=\{0\}$. This delivers the short exact sequence of groups \eqref{eq:sesgr}.}
\end{proof}


\end{document}